\documentclass[10pt]{amsart}
\usepackage{amssymb,amsbsy,amsmath,amsfonts,times, amscd}
\usepackage{latexsym,euscript,exscale}
\usepackage{graphicx,color} \usepackage{amsthm}
\usepackage{epigraph}\usepackage{lmodern}
\usepackage{mathrsfs}\usepackage{cancel}
\usepackage[toc,page]{appendix}\usepackage[T1]{fontenc}
\usepackage[all,cmtip]{xy}\usepackage{mathtools}
\usepackage{enumerate}\usepackage{setspace}
\usepackage{helvet}
\usepackage{float}
		
\numberwithin{equation}{section}
\newtheorem{theorem}{Theorem}
\newtheorem{thm}[theorem]{Theorem}
\newtheorem{cor}[theorem]{Corollary}
\newtheorem{lemma}[theorem]{Lemma}
\newtheorem{prop}[theorem]{Proposition}
\theoremstyle{definition}

\newtheorem{problem}[theorem]{Problem}
\theoremstyle{remark}
\newtheorem{remark}[theorem]{Remark}

\numberwithin{theorem}{section}

\newcommand{\vertiii}[1]{{\left\vert\kern-0.25ex\left\vert\kern-0.25ex\left\vert #1 
    \right\vert\kern-0.25ex\right\vert\kern-0.25ex\right\vert}}

\newcommand {\M}{\mathbb M}

\newcommand{\eps}{\varepsilon}

\newcommand{\SB}{\text{SB}}

\newcommand{\Q}{\mathbb{Q}}
\newcommand{\N}{\mathbb{N}}
\newcommand{\R}{\mathbb{R}}

\begin{document}

\title[Asymptotic structure and coarse Lipschitz geometry.]{Asymptotic structure and coarse Lipschitz geometry of Banach spaces.}
\author{B. M. Braga }
\address{Department of Mathematics, Statistics, and Computer Science (M/C 249)\\
University of Illinois at Chicago\\
851 S. Morgan St.\\
Chicago, IL 60607-7045\\
USA}\email{demendoncabraga@gmail.com}
\date{}
\maketitle

\begin{abstract}
In this paper, we study the coarse Lipschitz geometry of Banach spaces with several asymptotic properties. Specifically, we look at asymptotic uniform smoothness and convexity, and several distinct Banach-Saks-like properties. Among other results, we  characterize the Banach spaces which are either coarsely or uniformly homeomorphic to $T^{p_1}\oplus \ldots \oplus T^{p_n}$, where each $T^{p_j}$ denotes the $p_j$-convexification of the Tsirelson space, for $p_1,\ldots,p_n\in (1,\ldots, \infty)$, and $2\not\in\{p_1,\ldots ,p_n\}$. We obtain applications to the coarse Lipschitz geometry of the $p$-convexifications of the Schlumprecht space, and some hereditarily indecomposable Banach spaces. We also obtain some new results on the linear theory of Banach spaces.
\end{abstract}

\section{Introduction.}\label{sectionintro}

In this paper, we study nonlinear embeddings and nonlinear equivalences between Banach spaces. For that, we look at a Banach space $(X,\|\cdot\|)$ as a metric space endowed with the metric $\|\cdot-\cdot\|$. Let $(M,d)$ and $(N,\partial)$ be metric spaces, and $f:M\to N$ be a map. For each $t\in[0,\infty)$, we define the \emph{expansion modulus of $f$} as 

$$\omega_f(t)=\sup\{\partial (f(x),f(y))\mid d(x,y)\leq t\}$$\hfill

\noindent  and the \emph{compression modulus of $f$} as

$$\rho_f(t)=\inf\{\partial (f(x),f(y))\mid d(x,y)\geq t\}.$$\hfill

\noindent We say that $f$ is a \emph{coarse map} if  $\omega_f(t)<\infty$, for all $t\in [0,\infty)$. If, in addition, $\lim_{t\to\infty}\rho_f(t)=\infty$, then $f$ is a \emph{coarse embedding}. We say that $f$ is a \emph{coarse equivalence} if $f$ is both a coarse embedding and \emph{cobounded}, i.e., $\sup_{y\in N}\partial (y,f(M))<\infty$.  The map $f$ is a \emph{uniform embedding} if  $\lim_{t\to 0_+}\omega_f(t)=0$ and  $\rho_f(t)>0$, for all $t\in (0,\infty)$. A surjective uniform embedding is called a \emph{uniform homeomorphism}.  If there exists $L>0$ such that $\omega_f(t)\leq Lt+L$, for all $t\in [0,\infty)$, then we call $f$ a \emph{coarse Lipschitz map}. If, in addition,  $\rho_f(t)\geq L^{-1}t-L$, for all $t\in [0,\infty)$, then $f$ is a \emph{coarse Lipschitz embedding}.

 A  uniformly continuous map $f:X\to N$ from a Banach space $X$ to a metric space $N$ is automatically  a coarse map (see \cite{K3}, Lemma 1.4). Similarly, $f:X\to M$ is a coarse map  if and only if it is a coarse Lipschitz map (see \cite{K3}, Lemma 1.4). Also, if two Banach spaces $X$ and $Y$ are coarsely equivalent (resp. uniformly homeomorphic) then  $X$ coarse Lipschitz embeds into $Y$ (see \cite{K3}, Proposition 1.5).

In these notes, we are mainly interested in  what kind of stability properties those notions of nonlinear embeddings and nonlinear equivalences may have, and we will mainly work with Banach spaces having some kind of asymptotic property. More specifically, we are concerned with  asymptotically  uniformly smooth Banach spaces, asymptotically uniformly convex Banach spaces, and Banach spaces having several different Banach-Saks-like properties (we refer to Section \ref{background} for precise definitions).

The following general question is a central problem when dealing with nonlinear embeddings between Banach spaces.

\begin{problem}\label{problem1}
Let $\mathcal{P}$ and $\mathcal{P}'$ be two classes of Banach spaces and $\mathcal{E}$ be a kind of nonlinear embedding between Banach spaces. If a Banach space $X$ $\mathcal{E}$-embeds into a Banach space $Y$ in $\mathcal{P}$, does it follow that $X$ is in $\mathcal{P}'$?  
\end{problem}

For example, if a separable Banach space $X$ coarse Lipschitz embeds into a super-reflexive Banach space, then $X$ is also super-reflexive (this follows from Proposition 1.6 of \cite{K3} and Theorem 2.4 of \cite{K3}, but it was first proved  for uniform equivalences in \cite{R}, Theorem 1A).  Another example was given by  M. Mendel and A. Naor in \cite{MN} (Theorem 1.11), where they showed that if a Banach space $X$ coarsely  embeds into a Banach space $Y$ with cotype $q$ and non trivial type, then $X$ has cotype $q+\eps$, for all $\eps>0$. 

If we look at  nonlinear equivalences between Banach spaces, the following is a central problem in the theory.

\begin{problem}\label{problem2}
Let $X$ be a Banach space and $\mathcal{E}$ be a kind of nonlinear equivalence  between Banach spaces. If a Banach space $Y$ is $\mathcal{E}$-equivalent to $X$, what can we say about the isomorphism type of $Y$? More precisely:

\begin{enumerate}[(i)]
\item Is the linear structure of $X$ determined by its $\mathcal{E}$-structure, i.e., if a Banach space $Y$ is $\mathcal{E}$-equivalent to $X$, does it follow that $Y$ is linearly isomorphic to $X$? 
\item Let $\mathcal{P}$ be a class of Banach spaces.  If $Y$ is $\mathcal{E}$-equivalent to $X$, does is follow that $Y$ is linearly isomorphic to $X\oplus Z$, for some Banach space $Z$ in $\mathcal{P}$? 
\end{enumerate}
\end{problem}

Along those lines, it was shown in \cite{JLS} (Theorem 2.1) that  the coarse (resp. uniform) structure of $\ell_p$ completely determines its linear structure, for any $p\in (1,\infty)$. For $p=1$,  we do not even know if the Lipschitz structure of  $\ell_1$ determines its linear structure.  N. Kalton and N. Randrianarivony proved in \cite{KR} (Theorem 5.4) that,  for any $p_1,\ldots, p_n\in (1,\infty)$ with $2\not\in \{p_1,\ldots,p_n\}$,  the linear structure of $\ell_{p_1}\oplus\ldots\oplus \ell_{p_n}$ is determined by its coarse (resp. uniform) structure (see also \cite{JLS}, Theorem 2.2). This problem is still open if $2\in\{p_1,\ldots,p_n\}$.

Let $T$ denote the Tsirelson space introduced by T. Fiegel and W. Johnson in \cite{FJ}. For each $p\in [1,\infty)$, let $T^p$ be the $p$-convexification of $T$ (see Subsection \ref{submixtsi} for definitions). W. Johnson, J. Lindenstrauss and G. Schechtman addressed Problem \ref{problem2}(ii) above by  proving  the following (see \cite{JLS}, Theorem 5.8): suppose that either $1<p_1<\ldots<p_n<2$ or $2<p_1<\ldots<p_n$ and set $X=T^{p_1}\oplus \ldots\oplus T^{p_n}$, then  a Banach space $Y$ is coarsely  equivalent (resp. uniformly homeomorphic) to $X$ if and only if $Y$ is linearly isomorphic to $X\oplus \bigoplus_{j\in F}\ell_{p_j}$, for some $F\subset \{1,\ldots, n\}$.

We now describe the organization and some of the results of this paper. Firstly, in order not to make this introduction too extensive, we will postpone some technical definitions for later as well as our more technical results. The reader will find  all the background and notation necessary for this paper in Section \ref{background}.

Along the lines of Problem \ref{problem1}, we prove the following in Section \ref{sectionstableemb}.

\begin{thm}\label{asympbanachsaksref}
Let $Y$ be a  reflexive asymptotically uniformly smooth  Banach space, and assume that a Banach space $X$ coarse Lipschitz embeds into $Y$. Then $X$ has the Banach-Saks property. 
\end{thm}

As the Banach-Saks property implies reflexivity, Theorem \ref{asympbanachsaksref} above is a strengthening of Theorem 4.1 of \cite{BKL}, where the authors showed that if a separable Banach space $X$ coarse 
Lipschitz embeds into a reflexive asymptotically uniformly smooth Banach space, then $X$ must be reflexive. As $T$ is a reflexive Banach space without the Banach-Saks property, Theorem \ref{asympbanachsaksref} gives us  the following new  corollary.

\begin{cor}
The Tsirelson space does not coarse Lipschitz embed into any reflexive asymptotically uniformly smooth Banach space. 
\end{cor}

In Section  \ref{sectionstableemb}, we also prove some results on the linear theory of Banach spaces. Precisely, we show that an asymptotically uniformly smooth Banach space $X$ must have the alternating Banach-Saks property (see Corollary \ref{asympbansaks}).  Using descriptive set theoretical arguments, we also show that the converse does not hold, i.e., that there are Banach spaces with the alternating Banach-Saks property which do not admit an asymptotically uniformly smooth renorming (see Proposition \ref{8978}).

In Section \ref{sectionconvsmooth}, we study coarse  embeddings $f:X\to Y$ between Banach spaces $X$ and $Y$ with specific asymptotic properties, and  obtain a general result on how  close to an affine map the compression modulus $\rho_f$ can be (see Theorem \ref{upperpart}). Precisely, E. Guentner and J. Kaminker introduced the following quantity in \cite{GuKa}: for Banach spaces $X$ and $Y$,  define $\alpha_Y(X)$ as the supremum of all $\alpha>0$  for which there exists a coarse embedding $f:X\to Y$ and $L>0$ such that   $\rho_f(t)\geq L^{-1}t^\alpha-L$, for all $t\geq 0$. We call $\alpha_Y(X)$ the \emph{compression exponent of $X$ in $Y$}. As a simple consequence of Theorem \ref{upperpart}, we obtain Theorem \ref{corcorQuan} below.

We denote by $S$ the Schlumprecht space introduced in \cite{S}, and, for each $p\in [1,\infty)$, we let $S^p$ be the $p$-convexification of $S$  and $T^p$ be the $p$-convexification of the Tsirelson space $T$ (see Subsection \ref{submixtsi} for definitions).

\begin{thm}\label{corcorQuan}
Let $1\leq p<q$. Then 

\begin{enumerate}[(i)]
\item $\alpha_{T^q}(T^p)\leq p/q$, and
\item $\alpha_{S^q}(S^p)\leq p/q$.
\end{enumerate}

\noindent In particular, $T^p$ (resp. $S^p$) does not coarse Lipschitz embed into $T^q$ (resp. $S^q$).
\end{thm}

The proof of Theorem \ref{corcorQuan} is asymptotically in nature, hence we obtain equivalent estimates for the compression exponent $\alpha_Y(X)$, where $X$ and $Y$ are  Banach spaces  satisfying some special asymptotic properties. In particular,   the spaces $T^q$ and $S^q$  can be replaced in Theorem \ref{corcorQuan} by $(\oplus_n E_n)_{T^q}$ and $(\oplus_n E_n)_{S^q}$, where $( E_n)_{n=1}^\infty$ is any sequence of finite dimensional Banach spaces. See Theorem \ref{upperpartQuan}, Theorem \ref{almostco} and Corollary \ref{memata} for  precise statements.

We also  apply our results to the hereditarily indecomposable Banach spaces $\mathfrak{X}^p$ defined by N. Dew in \cite{D}, and obtain that $\alpha_{\mathfrak{X}^q}(\mathfrak{X}^p)\leq p/q$, for $1<p<q$ (see Corollary \ref{herdind}).

In Section \ref{sectioncoarseembintosum}, we prove a general theorem regarding the non existence of coarse Lipschitz embeddings $X\to Y_1\oplus Y_2$, for Banach spaces $X,Y_1,Y_2$ with specific asymptotic properties (see Theorem \ref{jointlowerupper}). With that result in hands, we prove the following.

\begin{thm}\label{notell2}
Let $1\leq p_1<\ldots <p_n<\infty$, and $p\in[1,\infty)\setminus\{p_1,\ldots ,p_n\}$. Then neither $T^{p}$ nor $\ell_p$ coarse Lipschitz embed into $T^{p_1}\oplus \ldots \oplus T^{p_n}$.  In particular, $T^p$ does not coarse Lipschitz embed into $T^q$, for all $p,q\in [1,\infty)$ with $p\neq q$. 
\end{thm}

At last, we use Theorem \ref{notell2} in order to obtain the following characterization. 

\begin{thm}\label{aplicacaoporra}
Let $1<p_1<\ldots <p_n<\infty$ with $2\not\in\{p_1,\ldots ,p_n\}$. A Banach space $Y$ is coarsely equivalent (resp. uniformly homeomorphic) to $X=T^{p_1} \oplus\ldots \oplus T^{p_n}$ if and only if $Y$ is linearly isomorphic to $X\oplus \bigoplus_{j\in F}\ell_{p_j}$, for some $F\subset\{1,\ldots ,n\}$. 
\end{thm}

Clearly, Theorem \ref{aplicacaoporra} is a strengthening of Theorem 5.8 of \cite{JLS} mentioned above. However, just as in the case for $\ell_{p_1}\oplus \ldots\oplus \ell_{p_n}$, we still do not know whether the theorem above holds if $2\in \{p_1,\ldots,p_n\}$.

\section{Notation and background.}\label{background}
\subsection{Basic definitions.}

All the Banach spaces in these notes are assumed to be infinite dimensional unless otherwise stated. Let $X$  be a Banach space. We denote the closed unit ball of $X$ by $B_X$, and its unit sphere by $\partial B_X$. If $Y$ is also a Banach space, we write $X\cong Y$ if $X$ is linearly isomorphic to $Y$. Given a Banach space $X$ with norm $\|\cdot\|_X$, we simply write $\|\cdot\|$ as long as it is clear from the context to which space  the elements inside the norm belong to. A sequence $(x_n)_{n=1}^\infty$ in a Banach space $X$ is called \emph{semi-normalized} if it is bounded and $\inf_n\|x_n\|>0$.

Say $(e_n)_{n=1}^\infty$ is a basis for the Banach space $X$. For $x=\sum_{n=1}^\infty x_ne_n\in X$, we write $\text{supp}(x)=\{n\in\N\mid x_n\neq 0\}$. For all finite subsets $E,F\subset \N$, we write $E<F$ (resp. $E\leq F$) if $\max E<\min F$ (resp. $\max E\leq\min F$). We call a sequence $(y_n)_{n=1}^\infty$ in $X$ a \emph{block sequence of $(e_n)_{n=1}^\infty$} if $\text{supp}(y_n)<\text{supp}(y_{n+1})$, for all $n\in\N$.

Let $(X_n)_{n=1}^\infty$ be a sequence of Banach spaces. Let $\mathcal{E}=(e_n)_{n=1}^\infty$ be a $1$-unconditional basic sequence in a Banach space $E$ with norm $\|\cdot\|_E$. We define the sum $(\oplus _n X_n)_{\mathcal{E}}$ to be the space of sequences $(x_n)_{n=1}^\infty$, where $x_n\in X_n$, for all $n\in\N$, such that 

$$\|(x_n)_{n=1}^\infty\|\vcentcolon =\Big\|\sum_{n\in\N}\|x_n\|e_n\Big\|_E<\infty.$$\hfill

\noindent One can check that $(\oplus _n X_n)_\mathcal{E}$ endowed with the norm $\|\cdot\|$ defined above is a Banach space.  If the $X_n$'s are all the same, say $X_n=X$, for all $n\in\N$, we write $(\oplus X)_\mathcal{E}$. Also, if it is implicit what is the basis  $\mathcal{E}$ of the Banach space $E$  that we are working with, we write $(\oplus_n X_n)_E$.

\subsection{$p$-convex and $p$-concave Banach spaces.}\label{subseclower} Let $X$ be a Banach space with $1$-unconditional basis $(e_n)_{n=1}^\infty$, and let $p\in (1,\infty)$. We say that the basis $(e_n)_{n=1}^\infty$ is \emph{$p$-convex with convexity  constant $C$} (resp. \emph{$p$-concave with concavity constant $C$}), if

$$\Big\|\sum_{j\in\N}(|x^1_j|^p+\ldots +|x^k_j|^p)^{1/p}e_j\Big\|^p\leq C^p \sum_{n=1}^k\|x^n\|^p,$$\hfill

$$\Big(\text{resp. }\ \ C^p\Big\|\sum_{j\in\N}(|x^1_j|^p+\ldots +|x^k_j|^p)^{1/p}e_j\Big\|^p\geq \sum_{n=1}^k\|x^n\|^p\Big),$$\hfill

\noindent for all $x^1=\sum_{j=1}^\infty x^1_je_j,\ldots ,x^k=\sum_{j=1}^\infty x^k_je_j\in X$.  We say that the basis $(e_n)_{n=1}^\infty$  satisfies an \emph{upper $\ell_p$-estimate  with constant $C$} (resp. \emph{lower $\ell_p$-estimate with constant $C$}), if 

$$\|x_1+\ldots+x_k\|^p\leq C^p \sum_{n=1}^k\|x_n\|^p\ \ \text{\Big(resp. }C^p\|x_1+\ldots+x_k\|^p\geq  \sum_{n=1}^k\|x_n\|^p\text{\Big)},$$\hfill

\noindent for all $x_1,\ldots,x_k\in X$ with disjoint supports. Clearly, a $p$-convex (resp. $p$-concave) basis with constant $C$ satisfies an upper (resp. lower) $\ell_p$-estimate with constant $C$.

\subsection{$p$-convexification.}

Let $X$ be a Banach space with a $1$-unconditional basis $(e_n)_{n=1}^\infty$. For any $p\in[1,\infty)$, we define the \emph{$p$-convexification of $X$} as follows. Let 

$$X^{p}=\Big\{(x_n)_{n=1}^\infty\in \R^\N\mid x^p\coloneqq \sum_{n\in\N}|x_n|^pe_n\in X\Big\},$$\hfill

\noindent and endow $X^p$ with the norm $\|x\|_{p}=\|x^p\|^{1/p}$, for all $x\in X^{p}$. By abuse of notation, we denote by $(e_n)_{n=1}^\infty$ the sequence of coordinate vectors in $X^{p}$. It is clear that $(e_n)_{n=1}^\infty$ is a $1$-unconditional basis for $X^{p}$ and that $X^1=X$. Also, the triangle inequality gives us that $X^{p}$ is $p$-convex with constant $1$.

\subsection{Asymptotically $p$-uniformly smooth and convex spaces.}

Let $X$  be a Banach space. We define the \emph{modulus of asymptotic uniform smoothness of $X$} as

$$\overline{\rho}_X(t)=\sup_{x\in \partial B_X}\inf_{\text{dim}(X/E)<\infty}\sup_{ h\in  \partial B_E}\|x+th\|-1.$$\hfill

\noindent We say that $X$ is \emph{asymptotically uniformly smooth} if $\lim_{t\to 0_+}\overline{\rho}_X(t)/t=0$. If there exists $p\in (1,\infty)$ and $C>0$ such that $\overline{\rho}_X(t)\leq Ct^p$, for all $t\in [0,1]$, we say that $X$ is \emph{asymptotically $p$-uniformly smooth}. Every  asymptotically uniformly smooth Banach space is asymptotically $p$-uniformly smooth for some $p\in (1,\infty)$ (this was first proved in \cite{KOS} for separable Banach spaces, and later generalized for any Banach space in \cite{Ra},  Theorem 1.2).

Let $X$  be a Banach space. We define the \emph{modulus of asymptotic uniform convexity of $X$} as

$$\overline{\delta}_X(t)=\inf_{x\in \partial B_X}\sup_{\text{dim}(X/E)<\infty}\inf_{ h\in \partial B_E}\|x+th\|-1.$$\hfill

\noindent We say that $X$ is \emph{asymptotically uniformly convex} if $\overline{\delta}_X(t)>0$, for all $t>0$. If there exists $p\in (1,\infty)$ and $C>0$ such that $\overline{\delta}_X(t)\geq Ct^p$, for all $t\in [0,1]$, we say that $X$ is \emph{asymptotically $p$-uniformly convex}. 

The following proposition is straight forward.

\begin{prop}\label{lemmadgj}
Let $p\in (1,\infty)$ and let $X$ be a Banach space with a $1$-unconditional basis satisfying an upper $\ell_p$-estimate (resp. lower $\ell_p$-estimate) with constant $1$. Then $X$ is asymptotically $p$-uniformly smooth (resp. asymptotically $p$-uniformly convex).
\end{prop}

\subsection{Banach-Saks properties.}

A Banach space $X$ is said to have the \emph{Banach-Saks property} if every bounded sequence $(x_n)_{n=1}^\infty$ in $X$ has a subsequence $(x_{n_j})_{j=1}^\infty$ such that $(\frac{1}{k}\sum_{j=1}^k x_{n_j})_{k=1}^\infty$ converges. A Banach space $X$ is said to have the \emph{alternating Banach-Saks property} if every bounded sequence $(x_n)_{n=1}^\infty$ in $X$ has a subsequence $(x_{n_j})_{j=1}^\infty$ such that $(\frac{1}{k}\sum_{j=1}^k \eps_jx_{n_j})_{k=1}^\infty$ converges, for some $(\eps_j)_{j=1}^\infty\in \{-1,1\}^\N$. For a detailed study of this properties, we refer to \cite{Be}.

Let $p\in (1,\infty)$. A Banach space $X$ is said to have the \emph{$p$-Banach-Saks property} (resp. \emph{$p$-co-Banach-Saks property}), if for every semi-normalized weakly null sequence $(x_n)_{n=1}^\infty$ in $X$, there exists a subsequence $(x_{n_j})_{j=1}^\infty$ and $c>0$ such that

$$\|x_{n_1}+\ldots+x_{n_k}\|\leq ck^{1/p}\ \ \text{ (resp. }\ \ \|x_{n_1}+\ldots+x_{n_k}\|\geq ck^{1/p}\text{),}$$\hfill

\noindent for all $k\in\N$, and all $ k\leq n_1<\ldots<n_k$. 


The following is a combination of  Proposition 1.2, Proposition 1.3, and Proposition 1.6 of \cite{DGJ} (Proposition 1.6 of \cite{DGJ} only mentions the $p$-Banach-Saks property, but a straight forward modification of their proof gives us the result for the $p$-co-Banach-Saks property).

\begin{prop}\label{lemmakalton}
Let $p\in (1,\infty)$ and let $X$ be a Banach space. If $X$ asymptotically $p$-uniformly smooth (resp. asymptotically $p$-uniformly convex), then $X$ has the $p$-Banach-Saks property (resp. $p$-co-Banach-Saks property)
\end{prop}

\subsection{Tsirelson and Schlumprecht spaces.}\label{submixtsi}

Let $c_{00}$ denote the set of sequences of real numbers which are eventually zero, and let $\|\cdot\|_0$ be the max norm on $c_{00}$. We denote by $T$ the \emph{Tsirelson space} defined in \cite{FJ}, i.e.,  $T$ is the  completion of $c_{00}$ under the unique norm $\|\cdot\|$ satisfying

$$\|x\|=\max\Big\{ \|x\|_0, \frac{1}{2}\cdot \sup\Big( \sum_{j=1}^k\|E_j x\|\Big)\Big\},$$ \hfill

\noindent where the inner supremum above is taken over all finite sequences $(E_j)_{j=1}^k$ of finite subsets of $\N$ such that $k\leq E_1< \ldots<E_k$. Therefore, for each $p\in(1,\infty)$, the norm $\|\cdot\|_p$ of the  $p$-convexified Tsirelson space $T^p$ satisfies

$$\|x\|_{p}=\max\Big\{\|x\|_0,\frac{1}{2^{1/p}}\cdot\sup\Big(\sum_{j=1}^k\|E_jx\|_{p}^p\Big)^{1/p}\Big\},$$\hfill

\noindent where the inner supremum above is taken over all finite sequences $(E_j)_{j=1}^k$ of finite subsets of $\N$ such that $k\leq E_1< \ldots<E_k$ (see \cite{CS}, Chapter X, Section E).

As $T^p$ satisfies an upper $\ell_p$-estimate with constant $1$, it follows that $T^p$ is asymptotically $p$-uniformly smooth and it has the $p$-Banach-Saks property.  Also, $T^{p}$ has the $p$-co-Banach-Saks property. Indeed, let $(e_n)_{n=1}^\infty$ be the standard basis for $T^{p}$. If $(x_n)_{n=1}^\infty$ is a normalized block subsequence of $(e_n)_{n=1}^\infty$, then

$$2^{-1/p}k^{1/p}=2^{-1/p}\Big(\sum_{n=k}^{2k-1}\|x_n\|_p^p\Big)^{1/p}\leq \Big\|\sum_{n=k}^{2k-1} x_n\Big\|_p,$$\hfill

\noindent for all $k\in\N$. Therefore, as for any normalized weakly null sequence $(x_n)_{n=1}^\infty$ in $T^{p}$, one can find a block sequence $(y_n)_{n=1}^\infty$ which is equivalent to a subsequence of $(x_n)_{n=1}^\infty$, we conclude that $T^{p}$ has the $p$-co-Banach-Saks.

\begin{remark}
Let $p\in (1,\infty)$. Then $T^p$  does not contain $\ell_r$ for any $r\in [1,\infty)$ (this is shown in \cite{Jo2} for $T$, and the result for $T^p$ follows analogously). Similarly, by duality arguments, ${T^p}^*$ does not contain $\ell_r$ for any $r\in [1,\infty)$ (the reader can find more on $T^p$ and similar duality arguments in \cite{CS}).
\end{remark}

The \emph{Schlumprecht space} $S$ (see \cite{S}) is defined as the  completion of $c_{00}$ under the unique norm $\|\cdot\|$ satisfying

$$\|x\|=\max\Big\{\|x\|_0,\sup \Big(\frac{1}{\log_2(k+1)}\sum_{j=1}^k\|E_jx\|\Big)\Big\},$$\hfill

\noindent where the inner supremum above is taken over all finite sequences $(E_j)_{j=1}^k$ of finite subsets of $\N$ such that $E_1< \ldots<E_k$. Similarly as with the $p$-convexified Tsirelson space, the norm $\|\cdot\|_p$  of the $p$-convexified Schlumprecht space $S^p$ is given by 

$$\|x\|_{p}=\max\Big\{\|x\|_0,\sup \Big(\frac{1}{\log_2(k+1)}\sum_{j=1}^k\|E_jx\|_{p}^p\Big)^{1/p}\Big\},$$\hfill

\noindent where the inner supremum above is taken over all finite sequences $(E_j)_{j=1}^k$ of finite subsets of $\N$ such that $E_1< \ldots<E_k$ (see \cite{D}, page 59).

Similarly to $T^p$,  $S^{p}$ is asymptotically $p$-uniformly smooth and has the $p$-Banach-Saks property, for $p\in (1,\infty)$.

\subsection{Almost $p$-co-Banach-Saks property.}

Although $T^{p}$ has the $p$-co-Banach-Saks property, $S^{p}$ does not. However, $S^{p}$ satisfies a weaker property that will be enough for our goals.  Let $p\in (1,\infty)$. We say that a Banach space $X$ has the \emph{almost $p$-co-Banach-Saks property} if for every semi-normalized weakly null sequence $(x_n)_{n=1}^\infty$ in $X$ there exists a subsequence $(x_{n_j})_{j=1}^\infty$, and a sequence of positive numbers $(\theta_j)_{j=1}^\infty$ in $[1,\infty)$ such that $\lim_{j\to \infty}j^\alpha\theta^{-1}_j=\infty$, for all $\alpha>0$, and 

$$ \|x_{n_1}+\ldots +x_{n_k}\|\geq k^{1/p}\theta^{-1}_k,$$\hfill

\noindent for all $k\in\N$, and all $ k\leq n_1<\ldots <n_k$. Clearly, $S^{p}$ has the almost  $p$-co-Banach-Saks property, with $\theta_k=\log_2(k+1)^{1/p}$, for all $k\in\N$.

\section{Asymptotic uniform smoothness and the alternating Banach-Saks property.}\label{sectionstableemb}

In this section, we are going to show that asymptotically uniformly smooth Banach spaces must have the alternating Banach-Saks property (Corollary \ref{asympbansaks}), but the converse does not hold (see Proposition \ref{8978}). Also, we show that if a Banach space $X$ coarse Lipschitz embeds into a reflexive space $Y$ which is also asymptotically uniformly smooth, then $X$ must have the Banach-Saks property (Theorem \ref{asympbanachsaksref}). As any space with the Banach-Saks property is reflexive, this is a strengthening of Theorem 4.1 of \cite{BKL}, which says that, under the same hypothesis, $X$ must be reflexive.

\begin{prop}\label{pbansaks}
Let $X$ be a Banach space with the $p$-Banach-Saks property, for some $p\in (1,\infty)$, and assume that $X$ does not contain $\ell_1$. Then $X$ has the alternating Banach-Saks property. In particular, if $X$ is also reflexive, then $X$ has the Banach-Saks property.
\end{prop}

\begin{proof}
Assume $X$ does not have the alternating Banach-Saks property. Then,  there exist $\delta>0$ and a bounded sequence $(x_n)_{n=1}^\infty$ in $X$ such that, for all $k\in\N$, all $\eps_1,\ldots ,\eps_k\in\{-1,1\}$, and all $n_1<\ldots <n_k\in\N$, we have

\begin{align}\label{delta}
\Big\|\frac{1}{k}\sum_{j=1}^k\eps_jx_{n_j}\Big\|>\delta
\end{align}\hfill

\noindent (see \cite{Be}, Theorem 1, page 369). As $X$ does not contain $\ell_1$, by Rosenthal's $\ell_1$-theorem (see \cite{R2}), we can assume that $(x_n)_{n=1}^\infty$ is weakly Cauchy. Hence,  the sequence $(x_{2n-1}-x_{2n})_{n=1}^\infty$ is weakly null. By  Equation (\ref{delta}),  it is also semi-normalized. Therefore, as $X$ has the $p$-Banach-Saks property, by taking a subsequence if necessary, we have that

$$\Big\|\sum_{j=1}^{k}(x_{n_{2j-1}}-x_{n_{2j}})\Big\|\leq ck^{1/p},$$\hfill

\noindent for all $k\in\N$, and some constant $c>0$ independent of $k$. By Equation (\ref{delta}), we get that

$$\delta<\Big\|\frac{1}{2k}\sum_{j=1}^{2k}(-1)^{j+1}x_{n_j}\Big\|\leq \frac{c}{2}k^{{1/p}-1}.$$\hfill

\noindent  As this holds for all $k\in\N$, and $p>1$, if we let $k\to \infty$, we get  that $\delta=0$, which is a contradiction.

For reflexive spaces, the alternating Banach-Saks property and the Banach-Saks property are equivalent (see \cite{Be}, Proposition 2), so the last statement of the proposition  follows.
\end{proof}

\begin{cor}\label{asympbansaks}
Let $X$ be an asymptotically uniformly smooth Banach space. Then $X$ has the alternating Banach-Saks property. In particular, if $X$ is also reflexive, then $X$ has the Banach-Saks property.
\end{cor}

\begin{proof}
As $X$ is asymptotically uniformly smooth, $X$ cannot contain $\ell_1$. Therefore, we only need to notice that $X$ has the $p$-Banach-Saks property, for some $p\in(1,\infty)$, and apply Proposition \ref{pbansaks}. By Theorem 1.2 of \cite{Ra}, $X$ is asymptotically $p$-uniformly smooth, for some $p\in(1,\infty)$. Therefore, by Proposition \ref{lemmakalton} above, we have that $X$ has the $p$-Banach-Saks property, so we are done.
\end{proof}

For each $k\in \N$ and each infinite subset $\M\subset \N$,  we define $G_k(\M)$ as the set of all subsets of $\M$ with $k$ elements. We write $\bar{n}=(n_1,\ldots ,n_k)\in G_k(\M)$ always in an increasing order, i.e., $n_1<\ldots <n_k$. We define a metric $d=d_k$  on $G_k(\M)$ by letting

$$d(\overline{n},\overline{m})=|\{j\mid n_j\neq m_j\}|,$$\hfill

\noindent for all $\overline{n}=(n_1,\ldots ,n_k),\overline{m}=(m_1,\ldots ,m_k)\in G_k(\M)$. 

The following will play an important role in many of the results in this paper. This result was proved in \cite{KR}, Theorem 4.2 (see also Theorem 6.1 of \cite{KR}).

\begin{thm}\label{KR4.2}
Let $p\in (1,\infty)$, and let $Y$ be a reflexive asymptotically $p$-uniformly smooth Banach space. There exists $K>0$ such that, for all infinite subset $\M\subset \N$, all $k\in\N$, and  all  bounded map $f: G_k(\M)\to Y$, there exists an infinite subset $\M'\subset \M$ such that 

$$\text{diam}(f(G_k(\M')))\leq K\text{Lip}(f) k^{1/p}.$$\hfill
\end{thm}

\begin{proof}[Proof of Theorem \ref{asympbanachsaksref}]
Let $f:X\to Y$ be a coarse Lipschitz embedding. Pick $C>0$  so that $\omega_f(t)\leq Ct+C$, $\rho_f(t)\geq C^{-1}t-C$, for all $t\geq 0$. Assume that $X$ does not have the Banach-Saks property. By \cite{Be}, page 373,  there exists $\delta>0$ and a sequence $(x_n)_{n=1}^\infty$ in $B_X$ such that, for all $k\in\N$, and all $n_1<\ldots <n_{2k}\in\N$,  we have that

$$ \Big\|\frac{1}{2k}\sum_{j=1}^{k}(x_{n_j}-x_{n_{k+j}})\Big\|\geq \delta.$$\hfill

For each $k\in\N$, define $\varphi_k:G_k(\N)\to X$ by setting $\varphi_k(n_1,\ldots ,n_k)=x_{n_1}+\ldots +x_{n_k}$, for all $(n_1,\ldots ,n_k)\in G_k(\N)$. Therefore,  $\text{diam} (\varphi_k(G_k(\M)))\geq 2k\delta$, and we have that $\text{diam} (f\circ \varphi_k(G_k(\M)))\geq 2k\delta C^{-1}-C$, for all $k\in\N$, and all infinite $\M\subset \N$.

As,  $\text{Lip}(\varphi_k)\leq  2$, we have that $\text{Lip}(f\circ \varphi_k)\leq  3C$.  As $Y$ is asymptotically uniformly smooth, there exists $p\in (1,\infty)$ for which $Y$ is asymptotically $p$-uniformly smooth (see \cite{Ra}, Theorem 1.2). By Theorem  \ref{KR4.2}, there exists $K=K(Y)>0$ and $\M\subset \N$ such that $\text{diam}(f\circ \varphi_k(G_k(\M)))\allowbreak \leq 3KCk^{1/p}$, for all $k\in\N$. We conclude that  

$$2k\delta C^{-1}-C\leq 3KC k^{1/p},$$\hfill

\noindent for all $k\in\N$. As $p>1$, this gives us a contradiction if we let $k\to \infty$.  
\end{proof}

The following was asked in \cite{GLZ}, Problem 2, and it remains open.

\begin{problem}\label{44344}
If a Banach space $X$ coarse Lipschitz embeds into a reflexive asymptotically uniformly smooth Banach space $Y$, does it follow that $X$ has an asymptotically uniformly smooth renorming?
\end{problem}

\begin{problem}\label{famunfemb}
Let $N$ be a metric space. We say that a family of metric spaces $(M_k)_{k=1}^\infty$ \emph{uniformly Lipschitz embeds into $N$} if there exists $C>0$ and Lipschitz embeddings $f_k:M_k\to N$ such that $\text{Lip}(f)\cdot\text{Lip}(f^{-1})<C$, for all $k\in \N$. Does the family $(G_k(\N),d)_{k=1}^\infty$ uniformly Lipschitz embed into any Banach space without an asymptotically uniformly smooth renorming?  
\end{problem}

As noticed in \cite{GLZ}, Problem 6,  a positive answer to Problem \ref{famunfemb} together with Theorem \ref{KR4.2} would give us a positive answer to Problem \ref{44344}.

It is worth noticing that the Banach-Saks property is not stable under uniform equivalences, hence, it is not stable under coarse Lipschitz isomorphisms either. Indeed, if $(p_n)_{n=1}^\infty$ is a sequence in $(1,\infty)$ converging to $1$, then $(\oplus_n\ell_{p_n})_{\ell_2}$ is uniformly equivalent to $(\oplus_n\ell_{p_n})_{\ell_2}\oplus \ell_1$ (see \cite{BeL}, page 244). The space $(\oplus_n\ell_{p_n})_{\ell_2}$ has the Banach-Saks property, while $(\oplus_n\ell_{p_n})_{\ell_2}\oplus \ell_1$ does not.

Let $\mathcal{G}(\N)$ denote the set of finite subsets of $\N$. We endow $\mathcal{G}(\N)$ with the metric $D$ given by

$$D(\overline{n},\overline{m})=|\overline{n}\Delta \overline{m}|,$$\hfill

\noindent for all $\overline{n}=(n_1,\ldots ,n_k),\overline{m}=(m_1,\ldots ,m_l)\in \mathcal{G}(\N)$, where $\overline{n}\Delta \overline{m}$ denotes the symmetric difference between the sets $\overline{n}$ and  $\overline{m}$.

\begin{prop}\label{Gklipunif}
 $\mathcal{G}(\N)$  Lipschitz embeds into any Banach space $X$ without the alternating Banach-Saks property. Moreover, for any $\eps>0$, the Lipschitz embedding $f:\mathcal{G}(\N)\to X$ can be chosen so that $\text{Lip}(f)\cdot\text{Lip}(f^{-1})<1+\eps$.
\end{prop}

\begin{proof}
By Theorem 1 of \cite{Be}, page 369, for all $\eta>0$, there exists a bounded sequence $(x_n)_{n=1}^\infty$ in $X$ such that, for all $k\in\N$, all $\eps_1,\ldots ,\eps_k\in\{-1,1\}$, and all $n_1<\ldots <n_k$, we have

$$1-\eta\leq \Big\|\frac{1}{k}\sum_{j=1}^k\eps_jx_{n_j}\Big\|\leq 1+\eta.$$\hfill

\noindent  Define $\varphi:\mathcal{G}(\N)\to X$ by setting $\varphi(n_1,\ldots ,n_k)=x_{n_1}+\ldots +x_{n_k}$, for all $(n_1,\ldots ,n_k)\in \mathcal{G}(\N)\setminus \{\emptyset\}$, and $\varphi(\emptyset)=0$. Then, we have that

$$(1-\eta)\cdot D(\overline{n},\overline{m})\leq\|\varphi_k(\overline{n})-\varphi_k(\overline{m})\|\leq  (1+\eta)\cdot D(\overline{n},\overline{m})$$\hfill

\noindent for all $\overline{n},\overline{m}\in \mathcal{G}(\N)$.
\end{proof}

\begin{problem}
If $X$ has the Banach-Saks property, does it follow that $\mathcal{G}(\N)$ does not Lipschitz embed into $X$? In other words, if $X$ is a reflexive Banach space, do we have that $\mathcal{G}(\N)$ Lipschitz embed into $X$ if and only if $X$ does not have the Banach-Saks property?
\end{problem}


 By Corollary \ref{asympbansaks} above,  any Banach space with an asymptotically uniformly smooth renorming has the alternating Banach-Saks property. To the best of our knowledge, there is no known example of a Banach space which has the alternating Banach-Saks property but does not admit an asymptotically uniformly smooth renorming. However, using descriptive set theoretical  arguments, one can show the existence of such spaces. Recall, $(X,\Omega)$ is called a \emph{standard Borel space} if $X$ is a set and $\Omega$ is a $\sigma$-algebra on $X$ which is the Borel $\sigma$-algebra associated to a Polish topology on $X$ (i.e., a topology generated by a complete separable metric). A subset $A\subset X$ is called \emph{analytic} if it is the image of a standard Borel space under a Borel map. We refer to \cite{Do} and \cite{Br}, Section 2, for more details  on the descriptive set theory of separable Banach spaces.

Let $C[0,1]$ be the space of continuous real-valued functions on $[0,1]$ endowed with the supremum norm. Let 

$$\SB=\{X\in C[0,1]\mid X\text{ is a closed linear subspace}\},$$\hfill

\noindent and endow $\SB$ with the Effros-Borel structure, i.e., the $\sigma$-algebra generated by

$$\{X\in \SB\mid X\cap U\neq \emptyset\},\ \ \text{for} \ \  U\subset C[0,1]\ \ \text{open.}$$\hfill

\noindent This makes $\SB$ into  a standard Borel space and, as $C[0,1]$ contains isometric copies of every separable Banach space, $\SB$ can be seen as a coding set for the class of all separable Banach spaces. Therefore, we can talk about  Borel and analytic classes of separable Banach spaces. 

By \cite{Br}, Theorem 17, the subset $\text{ABS}\subset \SB$ of Banach spaces with the alternating Banach-Saks is  not analytic. On the other hand, letting $\text{AUS}=\{X\in\SB\mid X\text{ is}\ \allowbreak\text{asymptotically}\ \allowbreak \text{uniformly smooth}\},$ we have

$$X\in \text{AUS}\Leftrightarrow \forall \eps\in \Q_+ \exists \delta\in \Q_+ \forall t\in \Q_+ \Big(t<\delta\Rightarrow \overline{\rho}_X(t)<\eps t\Big).$$\hfill

\noindent As $\{X\in \SB\mid \text{dim}(C[0,1]/X)<\infty\}$ is Borel, it is easy to check that the condition $A(t,\eps)\subset \SB$ given by

$$X\in A(t,\eps)\Leftrightarrow  \overline{\rho}_X(t)<\eps t$$\hfill

\noindent defines an  analytic subset of $\SB$ (for similar arguments, we refer to \cite{Do}, Chapter 2, Section 2.1). So,  $\text{AUS}$ must be analytic. Hence, letting $\text{AUSable}\subset \SB$ be the subset of Banach spaces with an asymptotically uniformly smooth renorming, we have that

$$X\in \text{AUSable}\Leftrightarrow \exists Y\in \text{AUS} \ \ \text{such that}\ \ X\cong Y.$$\hfill

\noindent As the isomorphism relation in $\SB\times \SB$ forms an analytic set (see \cite{Do}, page 11), it follows that $\text{AUSable}$ is analytic. This discussion together with Corollary \ref{asympbansaks} gives us the following.

\begin{prop}\label{8978}
$\text{AUSable}\subsetneq \text{ABS}$. In particular,   there exist separable Banach spaces with the alternating Banach-Saks property which do not admit an asymptotically uniformly smooth renorming.
\end{prop}


\section{Asymptotically $p$-uniformly convex/smooth spaces.}\label{sectionconvsmooth}

In this section, we will use results from \cite{KR} in order to obtain some restrictions on coarse embeddings $X\to Y$, where the spaces $X$ and $Y$ are assumed to have some asymptotic properties  (see Theorem \ref{upperpart}). We   obtain restrictions on the existence of coarse embeddings between the convexified Tsirelson spaces (Theorem \ref{corcorQuan}(i)),  convexified Schlumprecht spaces (Theorem \ref{corcorQuan}(ii)), and some specific hereditarily indecomposable spaces  introduced in \cite{D} (Corollary \ref{herdind}).

\begin{thm}\label{upperpart}
Let $p,q\in (1,\infty)$. Let $X$ be an infinite dimensional Banach space with the $p$-co-Banach-Saks property and  not containing $\ell_1$. Let $Y$ be a reflexive asymptotically $q$-uniformly smooth Banach space. Then, there exists  no coarse embedding $f:X\to Y$ such that

$$\limsup_{k\to\infty} \frac{\rho_f(k^{1/p})}{k^{1/q}}=\infty.$$\hfill
\end{thm}


\begin{proof}
Let $f:X\to Y$ be a coarse embedding. So, there exists $C>0$, such that $\omega_f(t)\leq Ct+C$, for all $t>0$. As $X$ does not contain $\ell_1$, by Rosenthal's $\ell_1$-theorem, we can pick a normalized weakly null sequence $(x_n)_{n=1}^\infty$ in $X$, with $\inf_{n\neq m}\|x_n-x_m\|>0$. For each $k\in\N$, define a map $\varphi_k: G_k(\N)\to X$ by letting

$$\varphi_k(n_1,\ldots ,n_k)=x_{n_1}+\ldots +x_{n_k},$$\hfill

\noindent for all $(n_1,\ldots ,n_k)\in G_k(\N)$. So, $\varphi_k$ is a bounded map.

If $d((n_1,\ldots ,n_k),(m_1,\ldots ,m_k))\leq 1$, then $\|\sum_{j=1}^k x_{n_j}-\sum_{j=1}^k x_{m_j}\|\leq 2$. So, $\text{Lip}(f\circ\varphi_k)\leq 3C$. By Theorem  \ref{KR4.2}, there exists $K=K(Y)>0$ and an infinite subset $\M_k\subset \N$ such that 

$$\text{diam}(f\circ\varphi_k(G_k(\M_k)))\leq 3KCk^{1/q}.$$\hfill

Without loss of generality, we may assume that $\M_{k+1}\subset \M_k$, for all $k\in\N$. Let $\M\subset \N$ diagonalize the sequence $(\M_k)_{k=1}^\infty$, say $\M=(n_j)_{j=1}^\infty$. If a sequence $(y_n)_{n=1}^\infty$ is weakly null, so is $(y_{2n-1}-y_{2n})_{n=1}^\infty$. Therefore, using the fact that  $X$ has the $p$-co-Banach-Saks property to the weakly null sequence $(x_{n_{2j-1}}-x_{n_{2j}})_{j=1}^\infty$, we get that there exists $c>0$  such that, for all $k\in\N$, there exists $ m_1<\ldots <m_{2k}\in \M_k$, such that

$$\Big\|\sum_{j=1}^k(x_{m_{2j-1}}-x_{m_{2j}})\Big\|\geq ck^{1/p}.$$\hfill

\noindent  Therefore, we have that $\text{diam}(\varphi_k(G_k(\M_k)))\geq ck^{1/p}$, which implies that $\text{diam}(f \circ\varphi_k(G_k(\M_k)))\geq \rho_f(ck^{1/p})$, for all $k\in\N$. So,

$$\rho_f(ck^{1/p})\leq 3KCk^{1/q},$$\hfill

\noindent for all $k\in\N$. Therefore, if $\limsup_{k\to\infty} \rho_f(k^{1/p})k^{-1/q}=\infty$, we get a contradiction.
\end{proof}

\begin{remark}\label{remell1} Let $X$ be any Banach space containing a sequence $(x_n)_{n=1}^\infty$ which  is asymptotically $\ell_1$, i.e., there exists $L>0$ such that, for all $m\in\N$, there exists $k\in\N$ such that $(x_{n_j})_{j=1}^m$ is $L$-equivalent to $(e_j)_{j=1}^m$, for all $k\leq n_1<\ldots <n_m\in\N$, where $(e_j)_{j=1}^\infty$ is the standard $\ell_1$-basis. Then, proceeding exactly as above, we can show that there exists no coarse embedding $f:X\to Y$ such that 

$$\limsup_{k\to \infty}\frac{\rho_f(k)}{k^{1/q}}=\infty,$$\hfill

\noindent  where $q\in (1,\infty)$ and $Y$ is a reflexive asymptotically $q$-uniformly smooth Banach space. 
\end{remark}

Let $X$ and $Y$ be Banach spaces. We define $\alpha_Y(X)$ as the supremum of all $\alpha>0$  for which there exists a coarse embedding $f:X\to Y$ and $L>0$ such that

$$L^{-1}\|x-y\|^\alpha-L\leq \|f(x)-f(y)\|,$$\hfill

\noindent for all $x,y\in X$. We call $\alpha_Y(X)$ the \emph{compression exponent of $X$ in $Y$}, or the \emph{$Y$-compression of $X$}. If, for all $\alpha>0$, no such $f$ and $L$ exist, we set $\alpha_Y(X)=0$. As $\omega_f$ is always bounded by an affine map (as $X$ is a Banach space), it follows  that $\alpha_Y(X)\in [0,1]$.  Also, $\alpha_Y(X)=0$ if $X$ does not coarsely embed into $Y$.

The quantity $\alpha_Y(X)$ was first introduced by E. Guentner and J. Kaminker in \cite{GuKa}. For a detailed study of $\alpha_{\ell_q}(\ell_p)$, $\alpha_{L_q}(\ell_p)$, $\alpha_{\ell_q}(L_p)$, and $\alpha_{L_q}(L_p)$, where $p,q\in (0,\infty)$, we refer to \cite{B}. 

Using this terminology, let us reinterpret Theorem \ref{upperpart}.

\begin{thm}\label{upperpartQuan}
Let $1<p<q$. Let $Y$ be a reflexive asymptotically $q$-uniformly smooth Banach space. The following holds.

\begin{enumerate}[(i)]
\item If $X$ contains a sequence which is asymptotically $\ell_1$, then $\alpha_Y(X)\leq 1/q$.
\item If $X$ is an infinite dimensional Banach space  with the $p$-co-Banach-Saks property and not containing $\ell_1$, then $\alpha_Y(X)\leq p/q$.
\end{enumerate}

\noindent In particular, $X$ does not coarse Lipschitz embed into $Y$. 
\end{thm}

\begin{proof}
(ii) Let $L>0$ and $f:X\to Y$ be a coarse embedding such that $\rho_f(t)\geq L^{-1}t^\alpha-L$, for all $t>0$. By Theorem \ref{upperpart}, we must have 

$$\limsup_{k\to\infty}k^{\alpha/p-1/q}L^{-1}-Lk^{-1/q}<\infty.$$ \hfill

\noindent Therefore, $\alpha/p-1/q\leq 0$, and the result follows.

(i) This follows from Remark \ref{remell1} and the same reasoning as item (ii) above.
\end{proof}

Notice that $Y$ being reflexive in Theorem \ref{upperpartQuan}  cannot be removed. Indeed,  $c_0$ contains a Lipschitz copy of any separable metric space (see \cite{Ah}), and it is also asymptotically  $q$-uniformly smooth, for any $q\in (1,\infty)$.

\begin{cor}
Let $1<p<q$. Let $X$ be asymptotically $p$-uniformly convex, and $Y$ be reflexive and asymptotically $q$-uniformly smooth. Then $\alpha_{Y}(X)\leq p/q$.
\end{cor}

Asking the Banach space $X$ to have the $p$-co-Banach-Saks property in Theorem \ref{upperpartQuan} is actually too much, and we can weaken this condition by only requiring $X$ to have the almost $p$-co-Banach-Saks property. Precisely, we have the following.

\begin{thm}\label{almostco}
Let $1<p<q$. Let $X$ be an infinite dimensional Banach space  with the almost $p$-co-Banach-Saks property. Let $Y$ be a reflexive asymptotically $q$-uniformly smooth Banach space. Then $\alpha_Y(X)\leq p/q$. In particular, $X$ does not coarse Lipschitz embed into $Y$.
\end{thm}

\begin{proof}
Let $f:X\to Y$ be a coarse embedding and pick $C>0$ such that $\omega_f(t)\leq Ct+C$, for all $t\geq 0$. If $X$ contains $\ell_1$,  the result follows from Theorem \ref{upperpartQuan}(i). If $X$ does not contain $\ell_1$, we can pick a normalized weakly null sequence $(x_n)_{n=1}^\infty$ in $X$, with $\inf_{n\neq m}\|x_n-x_m\|>0$. By taking a subsequence of $(x_n)_{n=1}^\infty$ if necessary, pick  $(\theta_k)_{k=1}^\infty$ as in the definition of the almost $p$-co-Banach-Saks property. Define $\varphi_k:G_k(\N)\to X$ by letting $\varphi_k(n_1,\ldots ,n_k)=x_{n_1}+\ldots +x_{n_k}$, for all $(n_1,\ldots ,n_k)\in G_k(\N)$.

 Following the proof  of Theorem \ref{upperpart}, we get  that 
 
$$\rho_f(k^{1/p}\theta_k^{-1})\leq 3KC k^{1/q},$$\hfill

\noindent  for all $k\in\N$. Let $L>0$ and $\alpha>0$ be such that $\rho_f(t)\geq L^{-1}t^\alpha-L$, for all $t>0$. Then, 

$$k^{\alpha/p-1/q}\theta_k^{-\alpha}L^{-1}\leq 4KC ,$$\hfill

\noindent for big enough $k\in\N$. As  $\lim_{k\to \infty} k^\beta\theta^{-\alpha}_k=\infty$, for all $\beta>0$, we must have that $\alpha/p -1/q\leq 0$.
\end{proof}

\begin{remark}\label{remell2}
Let $(x_n)_{n=1}^\infty$ be a bounded sequence in a Banach space $X$ with the following property: there exists a sequence of positive reals $(\theta_j)_{j=1}^\infty$ in $[1,\infty)$ such that $\lim_{j\to\infty} j^\alpha \theta_j^{-1}=\infty$, for all $\alpha>0$, and

\begin{equation}\label{estrela}
k\theta^{-1}_k\leq \|\pm x_{n_1}+\ldots +\pm x_{n_k}\|,\tag{$*$}
\end{equation}\hfill

\noindent for all $n_1<\ldots< n_k\in \N$. The proof of Theorem \ref{almostco}  gives us that $\alpha_Y(X)\leq 1/q$, for any  reflexive asymptotically $q$-uniformly smooth Banach space  $Y$, with $q>1$. 
\end{remark}

Let $q>1$,  and let $(E_n)_{n=1}^\infty$  be a sequence of finite dimensional Banach spaces.  Let $\mathcal{E}$ be a $1$-unconditional basic sequence. Notice that, if $\mathcal{E}$ generates a reflexive asymptotically $q$-uniformly smooth Banach space, then $(\oplus_n E_n)_{\mathcal{E}}$ is also reflexive and asymptotically $q$-uniformly smooth. Hence, Theorem \ref{upperpartQuan} and  Theorem \ref{almostco} gives us the following corollary.

\begin{cor}\label{memata}
Let $1<p<q$, and let $(E_n)_{n=1}^\infty$  be a sequence of finite dimensional Banach spaces. Let $\mathcal{E}$ be a $1$-unconditional basic sequence generating a reflexive asymptotically $q$-uniformly smooth Banach space. The following holds.

\begin{enumerate}[(i)]
\item If $X$ contains a sequence with Property $($\ref{estrela}$)$, then $\alpha_{(\oplus_nE_n)_\mathcal{E}}(X)\leq 1/q$.
\item If $X$ is an infinite dimensional Banach space  with the almost $p$-co-Banach-Saks property, then $\alpha_{(\oplus_nE_n)_\mathcal{E}}(X)\leq p/q$.
\end{enumerate}

\noindent In particular, $X$ does not coarse Lipschitz embed into $(\oplus_nE_n)_\mathcal{E}$.
\end{cor}

\begin{proof}[Proof of Theorem \ref{corcorQuan}]
(i) As noticed in Subsection \ref{submixtsi}, $T^{p}$ has the $p$-co-Banach-Saks property, and is asymptotically $p$-uniformly smooth, for all $p\in (1,\infty)$. Therefore, as $T^{p}$ is reflexive (see \cite{OSZ}, Proposition 5.3(b)), for all $p\in [1,\infty)$, the result follows from Theorem \ref{upperpartQuan} (or Corollary \ref{memata}).

(ii) For any $p\in (1,\infty)$, $S^p$ has the almost $p$-co-Banach-Saks property and is asymptotically $p$-uniformly smooth. By Theorem 8 and Proposition 2(2) of \cite{CKKM}, $S^p$ is  reflexive, for all $p\in [1,\infty)$. So, the result follows from Corollary \ref{memata}.
\end{proof}

A Banach space $X$ is called \emph{hereditarily indecomposable} if none of its subspaces can be decomposed as a sum of two infinite dimensional Banach spaces. In Chapter 5 of \cite{D}, for each $p\in (1,\infty)$,  Dew constructed a hereditarily indecomposable space $\mathfrak{X}_p$ with a basis $(e_n)_{n=1}^\infty$ satisfying the following properties: (i) $\mathfrak{X}_p$ is reflexive, (ii) the base $(e_n)_{n=1}^\infty$  satisfies an upper $\ell_p$-estimate with constant $1$, and (iii) if $(x_n)_{n=1}^\infty$ is a block sequence of $(e_n)_{n=1}^\infty$, then, for all $n\in\N$,

$$\Big\|\sum_{j=1}^nx_j\Big\|\geq {f(n)^{-1/p}}\Big(\sum_{j=1}^n\|x_j\|^p\Big)^{1/p},$$\hfill

\noindent where $f:\N\to [0,\infty)$ is  a function such	 that, among other properties,  $\lim_{n\to \infty}n^\alpha f(n)^{-1}\allowbreak =\infty$, for all $\alpha>0$. In particular, $\mathfrak{X}_p$ has the almost $p$-co-Banach-Saks property, and it is asymptotically $p$-uniformly smooth. This, together with Theorem \ref{almostco}, gives us the following.

\begin{cor}\label{herdind}
 Let $1<p<q$. Then $\alpha_{\mathfrak{X}^q}(\mathfrak{X}^p)\leq p/q$. In particular, $\mathfrak{X}_p$ does not coarse Lipschitz embeds into $\mathfrak{X}_q$.
\end{cor}

\begin{problem}
Let $1\leq p<q$.  Does  $\alpha_{T^q}(T^p)=\alpha_{S^q}(S^p)=p/q$? If $p>1$, does $\alpha_{\mathfrak{X}^q}(\mathfrak{X}^p)= p/q$ hold?
\end{problem}

\begin{remark}\label{cotypeT} It is worth noticing that, if $p>\max\{q,2\}$, then $\alpha_{T^q}(T^p)=0$. Indeed, for all $r\geq 2$, $T^r$ has cotype $r+\eps$ for all $\eps>0$ (see \cite{DJT}, page 305). On the other hand, if $r<2$, then $T^r$ has cotype $2$. This follows from the fact that, for any $\eps>0$, $T^r$ has an equivalent norm satisfying a lower $\ell_{(r+\eps)}$-estimate (we explain this in the proof of Corollary \ref{notell2} below), then, by  Theorem 1.f.7 and Proposition 1.f.3(i) of \cite{LT}, $T^r$ has cotype $2$. Similarly, by Theorem 1.f.7 and Proposition 1.f.3(ii), $T^r$ has non trivial type, for all $r\in (1,\infty)$. By Theorem 1.11 of \cite{MN}, if a Banach space $X$ coarsely embeds into a Banach space $Y$ with non trivial type, then 

$$\inf\{q\in [2,\infty)\mid X\text{ has cotype }q\}\leq\inf \{q\in [2,\infty)\mid Y\text{ has cotype }q\}.$$\hfill

\noindent Therefore, we conclude that  $T^p$ does not coarsely embed into $T^q$, if $p>\max\{q,2\}$. So,  $\alpha_{T^q}(T^p)=0$.
\end{remark}

\begin{problem}
Let $1\leq q<p\leq 2$. What can we say about $\alpha_{T^q}(T^p)$?
\end{problem}

We finish this section with an  application of Theorem \ref{upperpartQuan}, Theorem \ref{almostco}, and Theorem 3.4 of \cite{AB}. By looking at the proof of Theorem 3.4 of  \cite{AB}, one can easily see that the authors proved a stronger result than the one stated in their paper. Precisely, the authors proved the following.

\begin{thm}\label{baudieralbiac}
Let $0<p<q$. There exist maps $(\psi_{j}:\R\to \R)_{j=1}^\infty$ such that, for all $x,y\in \R$,

$$A_{p,q}|x-y|^p\leq \max\{|\psi_{j}(x)-\psi_{j}(y)|^q\mid j\in\mathbb{N}\}$$\hfill

\noindent and 

$$\sum_{j\in\mathbb{N}}|\psi_{j}(x)-\psi_{j}(y)|^q\leq B_{p,q}|x-y|^p,$$\hfill

\noindent where $A_{p,q},B_{p,q}$ are positive constants.
\end{thm}

\begin{prop}\label{quantequa9}
Let $1\leq p<q$. There exists a map $f:T^p\to (\oplus T^q)_{T^q}$  which is simultaneously a coarse and a uniform embedding such that $\rho_f(t)\geq Ct^{p/q}$, for some $C>0$. In particular,  $\alpha_{(\oplus T^q)_{T^q}}(T^p)= p/q$.
\end{prop}

\begin{proof}
Let $(\psi_j)_{j=1}^\infty$, $A_{p,q}$, and $B_{p,q}$ be given by Theorem  \ref{baudieralbiac}. Define $f:T^p\to (\oplus T_q)_{T^q}$  by letting

$$f(x)=\Big((\psi_{j}(x_n)-\psi_{j}(0))_{j=1}^\infty\Big)_{n=1}^\infty,$$\hfill

\noindent for all $x=(x_n)_{n=1}^\infty\in T^p$. One can easily check that $f$ satisfies

$$A_{p,q}^{1/q}\|x-y\|^{p/q}\leq \|f(x)-f(y)\|\leq B_{p,q}^{1/q} \|x-y\|^{p/q},$$\hfill

\noindent for all $x,y\in T^p$.

As $T^q$ is $q$-convex, it is easy to see that $(\oplus T^q)_{T^q}$ is  asymptotic $q$-uniformly smooth. Hence, as  $(\oplus T^q)_{T^q}$ is reflexive, we conclude that $\alpha_{(\oplus T^q)_{T^q}}(T^p)= p/q$.
\end{proof}

\begin{cor}
$T$ strongly embeds into a super-reflexive Banach space.
\end{cor}

\begin{proof}
It is easy to check that $(\oplus T^2)_{T^2}$ is super-reflexive. Indeed, super-reflexivity is equivalent to a uniformly convex renorming. Hence, if $\mathcal{E}$ is a $1$-unconditional basis generating a super-reflexive space, and $X$ is a super-reflexive space, then so is $(\oplus X)_\mathcal{E}$ (see \cite{LT}, page 100). 
\end{proof}

Similarly as above, we get the following proposition.

\begin{prop}\label{quantequa9}
Let $1\leq p<q$. There exists a map $f:S^p\to  (\oplus S^q)_{S^q}$ which is simultaneously a coarse and a uniform embedding such that $\rho_f(t)\geq Ct^{p/q}$, for some $C>0$. In particular,   $\alpha_{(\oplus S^q)_{S^q}}(S^p)= p/q$.
\end{prop}

\section{Coarse Lipschitz embeddings into sums.}\label{sectioncoarseembintosum}

In this last section, we will be specially interested in the nonlinear geometry of the Tsirelson space and its convexifications. In order to obtain Theorem \ref{notell2}, we will prove a technical result on the  coarse Lipschitz non embeddability of certain Banach spaces into the direct sum of Banach spaces with certain $p$-properties (Theorem \ref{jointlowerupper}).  The main goal of this section is to  characterize the Banach spaces which are  coarsely (resp. uniformly) equivalent to $T^{p_1}\oplus \ldots \oplus T^{p_n}$,  for $p_1,\ldots,p_n\in (1,\ldots, \infty)$, and $2\not\in\{p_1,\ldots ,p_n\}$. \\

 Given $x,y\in X$, and $\delta>0$ the \emph{approximate midpoint between $x$ and $y$ with error $\delta$} is given by

$$\text{Mid}(x,y,\delta)=\{z\in X\mid \max\{\|x-z\|,\|y-z\|\}\leq 2^{-1}(1+\delta)\|x-y\|\}.$$\hfill

\noindent The following lemma is an asymptotic version of Lemma 1.6(i) of \cite{JLS} and Lemma 3.2 of \cite{KR}.

\begin{lemma}\label{weaknulllemma}
Let $X$ be an asymptotically $p$-uniformly smooth Banach space, for some $p\in(1,\infty)$. There exists $c>0$ such that, for all $x,y\in X$, all $\delta>0$, and all weakly null sequence $(x_n)_{n=1}^\infty$ in $B_X$,  there exists $n_0\in\N$ such that, for all $n>n_0$,  we have

$$u+\delta^{1/p}\|v\|x_n\in\text{Mid}(x,y,c\delta),$$\hfill

\noindent where $u=\frac{1}{2}(x+y)$, and $v=\frac{1}{2}(x-y)$.
\end{lemma}

\begin{proof}
By Proposition 1.3 of \cite{DGJ},  there exists $c>0$  such that, for all weakly null sequence $(x_n)_{n=1}^\infty$ in $B_X$, we have

$$\limsup_n\|x+x_n\|^p\leq \|x\|^p+c\cdot\limsup_n\|x_n\|^p.$$\hfill

\noindent Fix such sequence. As $\|x-(u+\delta^{1/p}\|v\|x_n)\|=\|v-\delta^{1/p}\|v\|x_n\|$, we get

$$\limsup_n \Big\|x-\Big(u+\delta^{1/p}\|v\|x_n\Big)\Big\|^p\leq (1+c\delta)\|v\|^p.$$\hfill

\noindent Therefore, as $(1+c\delta)^{1/p}<1+c\delta$, there exists $n_0\in\N$ such that $\|x-(u+\delta^{1/p}\|v\|x_n)\|\leq (1+c\delta)\|v\|$, for all $n>n_0$. Similarly, we can assume that $\|y-(u+\delta^{1/p}\|v\|x_n)\|\leq (1+c\delta)\|v\|$, for all $n>n_0$. 
\end{proof}

The following lemma is a simple modification of Lemma 3.3 of \cite{KR}, or Lemma 1.6(ii) of \cite{JLS}, so we omit its proof.

\begin{lemma}\label{punfconlemma}
Suppose $1\leq p<\infty$, and let $X$ be  Banach space with a $1$-unconditional basis $(e_n)_{n=1}^\infty$ satisfying a lower $\ell_p$-estimate with constant $1$. For all $x,y\in X$, and all $\delta>0$, there exists a compact subset $K\subset X$, such that

$$\text{Mid}(x,y,\delta)\subset K +2\delta^{1/p}\|v\|B_X,$$\hfill

\noindent where $u=\frac{1}{2}(x+y)$, and $v=\frac{1}{2}(x-y)$.
\end{lemma}

For each $s>0$, let 

$$\text{Lip}_s(f)=\sup_{t\geq s}\frac{ \omega_f(t)}{t}\ \ \text{ and} \ \ \text{Lip}_\infty(f)=\inf_{s>0}\text{Lip}_s(f).$$\hfill

\noindent  We will need the  following proposition, which can be found in \cite{KR} as Proposition 3.1.

\begin{prop}\label{lipinfty}
Let $X$ be a Banach space and $M$ be a metric space. Let $f:X\to M$ be a coarse map with $\text{Lip}_\infty(f)>0$. Then, for all $\eps,t>0$, and all $\delta\in (0,1)$, there exists $x,y\in X$ with $\|x-y\|>t$ such that

$$f(\text{Mid}(x,y,\delta))\subset \text{Mid}(f(x),f(y),(1+\eps)\delta).$$\hfill
\end{prop}

The following lemma will play the same role in our settings as  Proposition 3.5 did in \cite{KR}. 

\begin{lemma}\label{lemma}
Let  $1\leq q<p$. Let $X$ be an asymptotically $p$-uniformly smooth Banach space, and $Y$ be a Banach space with a $1$-unconditional basis satisfying a lower $\ell_q$-estimate with constant $1$. Let $f:X\to Y$ be a coarse map. Then, for any $t>0$, and any $\delta\in (0,1)$, there exists $x\in X$, $\tau>t$, and a compact subset $K\subset Y$ such that, for any weakly null sequence $(x_n)_{n=1}^\infty$ in $B_X$, there exists $n_0\in\N$ such that

$$f(x+\tau x_n)\in K+\delta\tau B_Y,\ \ \text{ for all }\ \ n>n_0.$$\hfill
\end{lemma}

\begin{proof}
If $\text{Lip}_\infty(f)=0$, then there exists $\tau>t$ such that $\text{Lip}_\tau(f)<\delta$. Hence, $\omega_f(\tau)<\delta \tau$, and the result follows by letting  $x=0$ and $K=\{f(0)\}$. Indeed, if  $z\in B_X$, we have

$$\|f(\tau z)-f(0)\|\leq\omega_f(\|\tau z\|)\leq \omega_f(\tau)\leq \delta \tau.$$\hfill

Assume $\text{Lip}_\infty(f)>0$. In particular,  $C=\text{Lip}_s(f)>0$, for some $s>0$. Let $c>0$ be given by Lemma \ref{weaknulllemma} applied to $X$ and $p$. As $q<p$, we can pick $\nu\in(0,1)$ such that $2C(2c)^{1/q}\nu^{1/q-1/p}<\delta$. By Proposition \ref{lipinfty}, there exists $u,v\in X$ such that $\|u-v\|>\max\{s,2t\nu^{-1/p}\}$ and 

$$f(\text{Mid}(u,v,c\nu))\subset \text{Mid}(f(u),f(v),2c\nu).$$\hfill

 Let $x=\frac{1}{2}(u+v)$, and $\tau=\nu^{1/p}\|\frac{1}{2}(u-v)\|$ (so $\tau>t$). Fix a weakly null sequence $(x_n)_{n=1}^\infty$ in $B_X$. Then, by Lemma \ref{weaknulllemma}, there exists $n_0\in\N$ such that $x+\tau x_n\in \text{Mid}(u,v,c\nu)$, for all $n>n_0$. So, 

$$f(x+\tau x_n)\subset f(\text{Mid}(u,v,c\nu))\subset\text{Mid}(f(u),f(v),2c\nu),$$\hfill

\noindent  for all $n>n_0$. Let  $K\subset Y$ be given by Lemma \ref{punfconlemma} applied to $Y$,  $f(u), f(v)\in Y$, and $2c\nu$. So,

$$\text{Mid}(f(u),f(v),2c\nu)\subset K+2(2c)^{1/q}\nu^{1/q}\frac{\|f(u)-f(v)\| }{2}B_Y.$$\hfill

\noindent As $\text{Lip}_s(f)=C$, and as $\|u-v\|>s$, we have $\|f(u)-f(v)\|\leq C\|u-v\|=2C\tau \nu^{-1/p}$. Hence, 

$$2(2c)^{1/q}\nu^{1/q}\frac{\|f(u)-f(v)\|}{2}\leq 2C(2c)^{1/q}\nu^{1/q-1/p}\tau<\delta \tau,$$\hfill

\noindent and we are done.
\end{proof}

\begin{remark}\label{equivnorm}
Lemma \ref{lemma} remains valid if we only assume that $X$ has an \emph{equivalent} norm with which $X$ becomes asymptotically $p$-uniformly smooth. Indeed, let $M\geq 1$ be such that $B_{(X,\|\cdot\|)}\subset M \cdot B_{(X,\vertiii{\cdot})}$. Fix $t>0$, and $\delta\in (0,1)$. Applying  Lemma \ref{lemma} to $(X,\vertiii{\cdot})$ with  $t'=M.t$ and $\delta'=\delta/M$, we obtain $x\in X$, $\tau'>t'$, and a compact set $K\subset Y$. The result now follows by letting $\tau=\tau'/M$. 
\end{remark}

\begin{thm}\label{jointlowerupper}
Let $1 \leq q_1<p<q_2$. Assume that 

\begin{enumerate}[(i)]
\item $X$ is an asymptotically $p$-uniformly smooth Banach space  with  the $p$-co-Banach-Saks property, and it does not contain $\ell_1$,
\item $Y_1$ is a Banach space with a $1$-unconditional basis satisfying a lower $\ell_{q_1}$-estimate with constant $1$, and
\item $Y_2$ is a reflexive asymptotically $q_2$-uniformly smooth Banach space.
\end{enumerate}

\noindent  Then $X$ does not coarse Lipschitz embed into $Y_1\oplus Y_2$.
\end{thm}

\begin{proof}
Let $Y_1\oplus_1 Y_2$ denote the space $Y_1\oplus Y_2$ endowed with the norm $\|(y_1,y_2)\|=\|y_1\|+\|y_2\|$, for all $(y_1,y_2)\in Y_1\oplus Y_2$. Assume $f=(f_1,f_2):X\to Y_1\oplus_1 Y_2$ is  a coarse Lipschitz embedding. As $f$ is a coarse Lipschitz embedding, there exists $C>0$ such that $\rho_f(t)\geq C^{-1}t-C$, and $\omega_{f_2}(t)\leq Ct+C$, for all $t>0$. 

Fix $k\in\N$, and $\delta\in (0,1)$. Then, by Lemma \ref{lemma}, there exists $\tau>k$, $x\in X$,  and a compact subset $K\subset Y_1$, such that, for any weakly null sequence $(y_n)_{n=1}^\infty$ in $B_X$, there exists $n_0\in\N$, such that

$$f_1(x+\tau y_n)\in K+\delta \tau B_{Y_1},$$\hfill

\noindent for all $n>n_0$.

As $X$ does not contain $\ell_1$, by Rosenthal's $\ell_1$-theorem, we can pick a normalized weakly null sequence $(x_n)_{n=1}^\infty$ in $X$, with $\inf_{n\neq m}\|x_n-x_m\|>0$.  As $X$ has the $p$-Banach-Saks property (Proposition \ref{lemmakalton}), there exists $c>0$ (independent of $k$) such that, by going to a subsequence if necessary, we have

$$\|x_{n_1}+\ldots +x_{n_k}\|\leq ck^{1/p},$$\hfill

\noindent for all $	n_1<\ldots <n_k\in \N$. Define a map $\varphi_{k,\delta}: G_k(\N)\to X$ by letting

$$\varphi_{k,\delta}(n_1,\ldots ,n_k)=x+\frac{\tau}{c} k^{-1/p}(x_{n_1}+\ldots +x_{n_k}),$$\hfill

\noindent for all $(n_1,\ldots ,n_k)\in G_k(\N)$. 

As $d((n_1,\ldots ,n_k),(m_1,\ldots ,m_k))\leq 1$ implies $\|\sum_{j=1}^k x_{n_j}-\sum_{j=1}^k x_{m_j}\|\leq 2$, we have that $\text{Lip}(f_2\circ\varphi_{k,\delta})\leq 2\tau C k^{-1/p}c^{-1}+C$. Therefore, by Theorem \ref{KR4.2}, there exists $\M_{k,\delta}\subset \N$ such that

\begin{align*}
\text{diam}(f_2\circ\varphi_{k,\delta}(G_k(\M_{k,\delta})))\leq  2K\tau  C k^{1/q_2-1/p}c^{-1}+KCk^{1/q_2},
\end{align*}\hfill

\noindent for some $K>0$ independent of $k$ and $\delta$.

Notice that,  if $(n^j_1,\ldots ,n^j_k)_{j=1}^\infty$ is a sequence in $G_k(\M_{k,\delta})$, with $n^j_k<n^{j+1}_1$, for all $j\in\N$, then $(x_{n^j_1}+\ldots +x_{n^j_k})_{j=1}^\infty$ is a weakly null sequence in $ck^{1/p}\cdot B_X$. Therefore,

$$f_1\circ \varphi_{k,\delta}(n^j_1,\ldots ,n^j_k)\in K+\delta\tau B_{Y_1},$$\hfill

\noindent for large enough $j$. This argument and standard Ramsey theory, gives us that, by 	passing to a  subsequence of $\M_{k,\delta}$,  we can assume that, for all $(n_1,\ldots ,n_k)\in G_k(\M_{k,\delta})$, 

$$f_1\circ \varphi_{k,\delta}(n_1,\ldots ,n_k)\in K+\delta\tau B_{Y_1}.$$\hfill

\noindent Therefore, as $K$ is compact, by passing to a further subsequence, we can assume that  $\text{diam}(f_1\circ\varphi_{k,\delta}(G_k(\M_{k,\delta})))\leq 3\delta \tau$ (see Lemma 4.1 of \cite{KR}).

We have shown that, for all $k\in\N$, and all $\delta\in (0,1)$, there exists a subsequence $\M_{k,\delta}\subset \N$ such that

\begin{align}\label{403}
\text{diam}(f\circ\varphi_{k,\delta}(G_k(\M_{k,\delta})))\leq  2K\tau C k^{1/q_2-1/p}c^{-1}+KCk^{1/q_2}+3\delta \tau.
\end{align}\hfill

\noindent We may assume that $\M_{k+1,\delta}\subset \M_{k,\delta}$, for all $k\in\N$, and all $\delta\in (0,1)$. For each $\delta\in (0,1)$, let $\M_\delta\subset \N$ diagonalize the sequence $(\M_{k,\delta})_{k=1}^\infty$.

As  $X$ has the $p$-co-Banach-Saks property, arguing similarly as in the proof of Theorem \ref{upperpart}, we get that  there exists $d>0$ (independent of $k$) such that, for all $k\in\N$, there exists $ n_1<\ldots <n_{2k}\in \M_{k,\delta}$, such that

$$ \Big\|\sum_{j=1}^k(x_{n_{2j-1}}-x_{n_{2j}})\Big\|\geq dk^{1/p}.$$\hfill

\noindent Therefore, $\text{diam}(\varphi_{k,\delta}(G_k(\M_\delta)))\geq \tau d/c$, which implies that 

\begin{align}\label{1}
\text{diam}(f \circ\varphi_{k,\delta}(G_k(\M_\delta)))\geq \tau d (cC)^{-1}-C,
\end{align}\hfill

\noindent for all $k\in\N$, and all $\delta\in (0,1)$. So, Equation (\ref{403}) and Equation (\ref{1}) give us that

$$\tau d (cC)^{-1}-C\leq 2K\tau C k^{1/q_2-1/p}c^{-1}+KCk^{1/q_2} +3\delta \tau.$$\hfill

\noindent for all $k\in\N$, and all $\delta\in (0,1)$. As $\tau>k$, this gives us that

$$ d (cC)^{-1}-Ck^{-1}\leq 2K C k^{1/q_2-1/p}c^{-1}+KCk^{1/q_2-1} +3\delta $$\hfill

\noindent for all $k\in\N$, and all $\delta\in (0,1)$. As $q_2>p>1$, by letting $k\to \infty$ and $\delta\to 0$, we get a contradiction.
\end{proof}

If $T=(T_1,T_2):X\to Y_1\oplus Y_2$ is a linear isomorphic embedding, then either $T_1:X\to Y_1$ or $T_2:X\to Y_2$ is not strictly singular, i.e., $T_i:X_0\to Y_i$ is a linear isomorphic embedding, for some infinite dimensional subspace $X_0\subset X$, and some $i\in\{1,2\}$. Is there an  analog of this result for coarse Lipschitz embeddings? Precisely, we ask the following.

\begin{problem}
Let $X$, $Y_1$ and $Y_2$ be Banach spaces and consider a coarse Lipschitz embedding $f=(f_1,f_2):X\to Y_1\oplus Y_2$. Is there an infinite dimensional subspace $X_0\subset X$ such that either $f_1:X_0\to Y_1$ or $f_2:X_0\to Y_2$ is a coarse Lipschitz embedding?
\end{problem}


We can now prove Theorem \ref{notell2}, which  will be essential in the proof of Theorem \ref{aplicacaoporra}.

\begin{proof}[Proof of Theorem \ref{notell2}]
Say $m\in \{1,\ldots ,n-1\}$ is such that $p\in(p_m,p_{m+1})$ (the other cases have analogous proofs). Then $(T^{p_{m+1}}\oplus \ldots \oplus T^{p_n})_{\ell_\infty}$ is reflexive (see \cite{OSZ}, Proposition 5.3(b)). Also, it is easy to see that $(T^{p_{m+1}}\oplus \ldots \oplus T^{p_n})_{\ell_\infty}$ is asymptotically $p_{m+1}$-uniformly smooth. By Theorem \ref{jointlowerupper}, it is enough to prove the following claim.\\

\textbf{Claim:} Fix $\eps>0$ such that $p_m+\eps<p$. $(T^{p_1}\oplus \ldots \oplus T^{p_m})_{\ell_{p_m}}$  can be renormed so that it has a $1$-unconditional basis satisfying a lower $\ell_{(p_m+\eps)}$-estimate with constant $1$.\\

For each $k\in\N$ and $p\in[1,\infty)$, denote by $P_k=P^p_k:T^{p}\to T^{p}$ the projection on the first $k$ coordinates, and let $Q_k=\text{Id}-P_k$. By Proposition 5.6 of \cite{JLS}, there exists $M\in [1,\infty)$ and $N\in\N$ such that $Q_N(T^{p_j})$ has an equivalent norm with $(p_j+\eps)$-concavity constant  $M$, for all $j\in\{1,\ldots ,m\}$ (precisely, the modified Tsirelson norm has this property, see \cite{CS} for definition). 

As the shift operator on the basis of $T^{p}$ is an isomorphism onto $Q_1(T^{p})$, we have that $T^{p}\cong Q_k(T^{p})$, for all $k\in\N$, and all $p\in [1,\infty)$.  Therefore,  it follows that $(T^{p_1}\oplus \ldots \oplus T^{p_m})_{\ell_{p_m}}$ has an equivalent norm  with $(p_m+\eps)$-concavity constant $M$. By Proposition 1.d.8 of \cite{LT}, we can assume that $M=1$. As a $q$-concave basis  with constant $1$ satisfies  a lower $\ell_q$-estimate with constant $1$, we are done.
\end{proof}

Before given the proof of Theorem \ref{aplicacaoporra}, we need a lemma. For that, we must introduce some natation. Let $p\in(1,\infty)$. A Banach space $X$ is said to be \emph{as. $\mathcal{L}_p$} if there exists $\lambda>0$ so that for every $n\in\N$ there is a finite codimensional subspace $Y\subset X$ so that every $n$-dimensional subspace of $Y$ is contained in a subspace of $X$ which is $\lambda$-isomorphic to $L_p(\mu)$, for some $\mu$. As noticed in \cite{JLS}, Proposition 2.4.a, an as. $\mathcal{L}_p$ space is super-reflexive. Also, the $p$-convexifications $T^p$ are as. $\mathcal{L}_p$ (see \cite{JLS}, page 440). 

The following lemma, although not explicitily written,  is contained in the proof of Proposition 2.7 of \cite{JLS}. For the convenience of the reader, we provide its proof here. 

\begin{lemma}\label{lemmaref}
Say $1<p_1<\ldots<p_n<\infty$ and $X=X^{p_1} \oplus\ldots \oplus X^{p_n}$, where $X^{p_j}$ is as. $\mathcal{L}_{p_j}$, for all $j\in \{1,\ldots,n\}$. Assume that  $Y$ is coarsely equivalent to $X$. 

\begin{enumerate}[(i)]
\item Then there exists a separable Banach space $W$ such that $Y\oplus W$ is Lipschitz equivalent to $\bigoplus_{j=1}^n(X^{p_j}\oplus L_{p_j})$.
\item Moreover, if $Y=Y^{p_1} \oplus\ldots \oplus Y^{p_n}$, where $Y^{p_j}$ is as. $\mathcal{L}_{p_j}$, for all $i\in \{1,\ldots,n\}$, then $\bigoplus_{j=1}^n(Y^{p_j}\oplus L_{p_j})$ is Lipschitz equivalent to $\bigoplus_{j=1}^n(X^{p_j}\oplus L_{p_j})$.
\end{enumerate}
\end{lemma}

\begin{proof}
First we need some definitions. Let $\mathcal{U}$ be an ultrafilter on $\N$, and $Z$ be a Banach space. Then we define the \emph{ultrapower of $Z$ with respect to $\mathcal{U}$} as $Z_\mathcal{U}=\{(z_n)_{n=1 }^\infty\in Z^\N \mid \sup_{i\in \N}\|z_n\|<\infty\}/\sim$, where $(z_n)_{n=1}^\infty\sim (y_n)_{n=1}^\infty$ if $\lim_{n\in \mathcal{U}}\|z_n-y_n\|=0$. $Z_\mathcal{U}$ is a Banach space with norm $\|[(z_n)_{n=1}^\infty]\|=\lim_{n\in\mathcal{U}}\|z_n\|$, where $(z_n)_{n=1}^\infty$ is a representative of the class $[(z_n)_{n=1}^\infty]\in Z_\mathcal{U}$. Notice that $z\in Z\mapsto [(z)_{n=1}^\infty]\in Z_\mathcal{U}$ is a linear isometric embedding. If  $Z$ is reflexive, $Z$ is $1$-complemented in  the ultrapower $Z_\mathcal{U}$ (where the projection is given by $[(z_n)_n]\in Z_\mathcal{U}\mapsto w\text{-}\lim_{n\in\mathcal{U}}z_n\in Z$), and we write $Z_\mathcal{U}=Z\oplus Z_{\mathcal{U},0}$. Also, we have that  $(Z\oplus E)_\mathcal{U}=Z_\mathcal{U}\oplus E_\mathcal{U}$. We can now prove the lemma.  For simplicity, let us assume that $n=2$.



(i) Let $\mathcal{U}$ be a nonprincipal ultrafilter on $\N$.  As $Y$ is coarsely equivalent to $X$,  $Y_\mathcal{U}$ is Lipschitz equivalent to $X_\mathcal{U}=X^{p_1}_\mathcal{U} \oplus X^{p_2}_\mathcal{U}$ (see \cite{K3}, proposition 1.6).  As the spaces $X^{p_j}_\mathcal{U}$ are reflexive, using the separable complementation property for reflexive spaces (see \cite{FJP}, Section 3), we can pick complemented  separable subspaces $W\subset Y_{\mathcal{U},0}$, and $X_{j,0}\subset X^{p_j}_{\mathcal{U},0}$, for $j\in \{1,2\}$, such that $Y\oplus W$ is Lipschitz equivalent to $(X^{p_1}\oplus X_{1,0})\oplus (X^{p_2}\oplus X_{2,0})$.     By enlarging $X_{j,0}$ and $W$, if necessary, we can assume that $X_{j,0}=L_{p_j}$, for $j\in\{1,2\}$ (this follows from Proposition 2.4.a of \cite{JLS}, Theorem I(ii) and Theorem III(b) of \cite{LR}). 

(ii) The same argument as why $X_{1,0}\oplus X_{2,0}$ can be enlarged so that $X_{1,0}\oplus X_{2,0}=L_{p_1}\oplus L_{p_2}$ gives us that $W$ can also be assumed to be $L_{p_1}\oplus L_{p_2}$.
\end{proof}

We can now prove Theorem \ref{aplicacaoporra}. As mentioned in Section \ref{sectionintro}, Theorem \ref{aplicacaoporra} was proved in \cite{JLS} (Theorem 5.8) for the cases $1<p_1<\ldots <p_n<2$ and $2<p_1<\ldots <p_n<\infty$. In our proof, Theorem \ref{notell2} will play a similar role as Corollary 1.7 of \cite{JLS} did in their proof. Also, we use ideas in the proof of Theorem 5.3 of \cite{KR} in order to unify the cases $1<p_1<\ldots <p_n<2$ and $2<p_1<\ldots <p_n<\infty$. In order to avoid an unnecessarily extensive proof, we will only present the parts of the proof that require Theorem \ref{notell2} above, and therefore are different from what can be found in the present literature.  \\

\noindent \emph{Sketch of the proof of Theorem \ref{aplicacaoporra}.}
By Proposition 5.7 of \cite{JLS}, $T^p$ is uniformly homeomorphic to $T^p\oplus \ell_p$, for all $p\in [1,\infty)$. So, the backwards direction follows. Let us prove the forward direction. As uniform homeomorphism implies coarse equivalence, it is enough to assume that  $Y$ is coarsely equivalent  to $X$. By Theorem \ref{notell2}, $Y$ does not contain $\ell_2$. Let $m\in\{1,\ldots ,n-1\}$ be such that $2\in (p_m,p_{m+1})$ (if such $m$ does not exist, the result simply follows from Theorem 5.8 of \cite{JLS}).\\

\textbf{Claim 1:} $X\oplus\bigoplus_{j=1}^nL_{p_j}$ and  $Y\oplus \bigoplus_{j=1}^nL_{p_j}$ are Lipschitz equivalent.\\

 By Lemma \ref{lemmaref}(i), there exists a separable Banach space $W$ so that  $Y\oplus W$ is Lipschitz equivalent to $\bigoplus_{j=1}^n (T^{p_j}\oplus L_{p_j})$.  Hence,  the image of $Y$ through this Lipschitz equivalence is the range of a Lipschitz projection in $\bigoplus_{j=1}^n (T^{p_j}\oplus L_{p_j})$. Therefore, by  Theorem 2.2 of \cite{HM}, we have that $Y$ is isomorphic to a complemented subspace of $\bigoplus_{j=1}^n (T^{p_j}\oplus L_{p_j})$. Let $A$ be this isomorphic embedding. For each $i\in\{m+1,\ldots ,n\}$, let $\pi_i:Y\to L_{p_i}$ be the composition of $A$ with the projection $\bigoplus_{j=1}^n (T^{p_j}\oplus L_{p_j})\to L_{p_i}$. As $Y$ does not contain $\ell_2$, $\pi_i$ factors through $\ell_{p_i}$ (see \cite{J}). Hence, $Y$ is isomorphic to a complemented subspace of 

$$\bigoplus_{j=1}^m(T^{p_j}\oplus L_{p_j})\oplus \bigoplus_{j=m+1}^n (T^{p_j}\oplus \ell_{p_j}).$$\hfill

  As $Z_1\coloneqq\bigoplus_{j=1}^m(T^{p_j}\oplus L_{p_j})$ and $Z_2\coloneqq\bigoplus_{j=m+1}^n (T^{p_j}\oplus \ell_{p_j})$ are totally incomparable (i.e., none of their infinite dimensional subspaces are isomorphic), $Y\cong Y_1\oplus Y_2$, where $Y_1$ and $Y_2$ are complemented subspaces of $Z_1$ and $Z_2$, respectively (see \cite{EW}, Theorem 3.5). Hence, $Y^*_1$ is complemented in $Z^*_1$. Notice that,  as $Y$ is coarsely equivalent to the super-reflexive space $X$, $Y$ is also super-reflexive (see \cite{R}, Theorem 1A). Hence, $Y_1$ is super-reflexive, and so is $Y_1^*$. As $Y_1$ has cotype $2$ (see Remark \ref{cotypeT}) and $Y_1^*$ has non trivial type (as $Y_1^*$ is super-reflexive), it follows that $Y_1^*$ has type $2$ (see the remark below Theorem 1 in \cite{P2}). So, $Y_1^*$ does not contain a copy of $\ell_2$. Indeed, otherwise $Y_1^*$ would contain a complemented copy of $\ell_2$ (see \cite{M}), contradicting that $Y_1$ does not contain a copy of $\ell_2$.
  
Proceeding similarly as above and using that $Y_1^*$ does not contain $\ell_2$, the main theorem of \cite{J} implies that $Y_1^*$ is isomorphic to a complemented subspace of $\bigoplus_{j=1}^m ({T^{p_j}}^*\oplus \ell_{\tilde{p}_j})$, where each $\tilde{p}_j$ is the conjugate of $p_j$ (i.e., $1/p_j+1/\tilde{p}_j=1$). Therefore, $Y_1$  embeds into $\bigoplus_{j=1}^m (T^{p_j}\oplus \ell_{p_j})$ as a complemented subspace. This gives us that $Y$ embeds into $\bigoplus_{j=1}^n (T^{p_j}\oplus \ell_{p_j})$ as a complemented subspace. 

As the spaces $(T^{p_j}\oplus \ell_{p_j})_{j=1}^n$ are totally incomparable, we can write $Y$ as $Y_{p_1}\oplus\ldots \oplus Y_{p_n}$, where each $Y_{p_j}$ is a complemented subspace of $ T^{p_j}\oplus \ell_{p_j}$ (see \cite{EW}, Theorem 3.5) and it is an as. $\mathcal{L}_{p_j}$ (see  \cite{JLS}, Lemma 2.5 and Proposition 2.7). By Lemma \ref{lemmaref}(ii), we have that $X\oplus\bigoplus_{j=1}^nL_{p_j}$ and  $Y\oplus \bigoplus_{j=1}^nL_{p_j}$ are Lipschitz equivalent.\\

\textbf{Claim 2:} There exists a quotient $W$ of $L_{p_1}\oplus\ldots \oplus L_{p_n}$ such that $Y\oplus W$ is isomorphic to $X\oplus \bigoplus_{j=1}^nL_{p_j}$. \\

The prove of Claim 2 is the same as the proof of the claim in Proposition 2.10 of \cite{JLS}, so we do not present it here. Let us assume the claim and finish the proof.  As $X$ does not contain any $\ell_s$, every operator of $X$ into $\oplus_{j=1}^nL_{p_j}$ is strictly singular (see \cite{KM}, Theorem II.2 and Theorem IV.1). Therefore, by \cite{EW} (or \cite{LT}, Theorem 2.c.13), $Y\cong Y_X\oplus Y_L$ and $W\cong W_X\oplus W_L$, where $Y_X$ and $W_X$ are complemented subspaces of $X$,  $Y_L$ and $W_L$ are complemented subspaces of  $\oplus_{j=1}^nL_{p_j}$, and $X\cong Y_X\oplus W_X$. Proceeding as in the proof of Claim 1 above, we get that $Y_L$ is complemented in $\oplus_{j=1}^n\ell_{p_j}$. So, $Y_L$ is either finite dimensional or  isomorphic to $\oplus_{j\in F}\ell_{p_j}$, for some $F\subset \{1,\ldots ,n\}$.


Let us show that $W_X$ is finite dimensional. Suppose this is not the case. As $W$ is a quotient of $\oplus_{j=1}^nL_{p_j}$, and $W_X$ is complemented in $W$, we have that $W^*_X$ embeds into $\oplus_{j=1}^nL_{\tilde{p}_j}$, where each $\tilde{p}_j$ is the conjugate of $p_j$. Therefore, it follows that $W^*_X$ must contain some $\ell_s$ (see \cite{KM}, Theorem II.2 and Theorem IV.1). As $W^*_X$ embeds into $X^*$, and $X^*$ does not contain any $\ell_s$, this gives us a contradiction. 

As $X\cong Y_X\oplus W_X$, and $\text{dim}(W_X)<\infty$, we  have that $\text{dim}(X/Y_X)<\infty$. Therefore, as $X$ is isomorphic to its hyperplanes, we conlude that $Y_X\cong X$. So, we are done.
\qed

\begin{problem}
Does Theorem \ref{aplicacaoporra} hold if $2\in \{p_1,\ldots,p_n\}$?
\end{problem}

\begin{problem}
What can we say if a Banach space $X$ is either coarsely or uniformly equivalent to the Tsirelson space $T$?
\end{problem}

\begin{remark}\label{APPENDIX}
It is worth noticing that, using Remark \ref{equivnorm} and adapting the proofs of Theorem 5.5 and Theorem 5.7 of \cite{KR} to our settings, one can show that $(\oplus T_p)_{T_q}$ does not coarse Lipschitz embed into $T_p\oplus T_q$, for all $p,q\in [1,\infty)$ with $p\neq q$.
\end{remark}

\noindent \textbf{Acknowledgments:} I would like to thank my adviser C. Rosendal for all the help and attention he gave to this paper.   I would also like to thank  Th. Schlumprecht for helpful conversations about the results in Section \ref{sectionstableemb}, and the anonymous referee for their comments and suggestions.

\section{Appendix: computations for Remark \ref{APPENDIX}.}\label{sectionsums}

In this appendix, we  prove some technical results that will, in particular, gives us the claim in Remark \ref{APPENDIX}. Precisely, we show the following.

\begin{thm}\label{cortsisum}
Let $p,q\in [1,\infty)$, with $p\neq q$. Then $(\oplus T_p)_{T_q}$ does not coarse Lipschitz embed into $T_p\oplus T_q$.
\end{thm}

Let $X$ be a sequence of Banach spaces, $\mathcal{E}=(e_n)_n$ be a $1$-unconditional basic sequence, and consider the sum $(\oplus X)_{\mathcal{E}}$. We call a sequence $(x_n)_n$ in $(\oplus X)_\mathcal{E}$ a \emph{block sequence in } $(\oplus X)_\mathcal{E}$ if there exists a strictly increasing sequence $(p_n)_n\in\N^\N$ such that  $x_n=\sum_{j=p_n+1}^{p_{n+1}}y_j$, for all $n\in\N$, where $y_{p_n+1}\in X_{p_{n}+1},\ldots, y_{p_{n+1}}\in X_{p_{n+1}}$. Notice that, if $\mathcal{E}$ is a weakly null sequence, then any bounded block sequence in $(\oplus X)_\mathcal{E}$ is weakly null. Indeed, this follows from the fact that a bounded sequence is weakly null if and only if all of its subsequences  has a convex block subsequence converging to zero in norm (see, for example, \cite{Br}, Lemma 19). 

We omit the proof of the following two lemmas, as they are  essentially the same as the proofs of Lemma \ref{weaknulllemma} and Lemma \ref{lemma}, respectively.

\begin{lemma}\label{weaknullsumlemma}
Let $\mathcal{E}$ be a weakly null $1$-unconditional basis generating an asymptotically $p$-uniformly smooth Banach space, for some $p\in(1,\infty)$. Let $X$ be a Banach space. Let $x,y\in (\oplus X)_\mathcal{E}$, and $\delta>0$. Then, if $(x_n)_n$ is a block sequence in  $B_{(\oplus X)_\mathcal{E}}$,  there exists $n_0\in\N$, such that, for all $n>n_0$,  we have

$$u+\delta^{1/p}\|v\|x_n\in \text{Mid}(x,y,\delta),$$\hfill

\noindent where $u=\frac{1}{2}(x+y)$, and $v=\frac{1}{2}(x-y)$.
\end{lemma}

\begin{lemma}\label{lemmasum}
Let  $1\leq q<p$. Let $\mathcal{E}$ be a weakly null $1$-unconditional basis  generating an asymptotically $p$-uniformly smooth Banach space, for some $p\in(1,\infty)$. Let $X$ be a  Banach space, and let $Y$  be a Banach space with a $1$-unconditional basis satisfying a lower $\ell_q$-estimate with constant $1$. Let $f:(\oplus X)_\mathcal{E}\to Y$ be a coarse map. Then, for any $t>0$, and any $\delta\in (0,1)$, there exists $x\in (\oplus X)_\mathcal{E}$, a compact subset $K\subset  Y$, and $\tau>t$, such that, for any  block sequence $(x_n)_n$ in $B_{(\oplus X)_\mathcal{E}}$, there exists $n_0\in\N$ such that
	
$$f(x+\tau x_n)\in K+\delta\tau B_{(\oplus X)_\mathcal{E}},\ \ \text{ for all }\ \ n>n_0.$$\hfill

\end{lemma}

\begin{thm}\label{oqueestoufazendodavida}
Let $1\leq q< p$ and $1<r< l$. Assume that

\begin{enumerate}[(i)]
\item $X$ is a Banach space  with the $r$-co-Banach-Saks property,  with   the $r$-Banach-Saks property, and not containing $\ell_1$,
\item  $\mathcal{E}$ is a weakly null normalized $1$-unconditional basis  generating an asymptotically $p$-uniformly smooth Banach space,
\item  $Y_1$ is a Banach space with a $1$-unconditional basis satisfying a lower $\ell_q$-estimate with constant $1$, and
\item $Y_2$ is a reflexive asymptotically $l$-uniformly smooth Banach space.
\end{enumerate}

\noindent Then $(\oplus X)_\mathcal{E}$ does not coarse Lipschitz embed into $Y_1\oplus Y_2$.  
\end{thm}

\begin{proof}
Let $f=(f_1,f_2):(\oplus X)_\mathcal{E}\to Y_1\oplus_1 Y_2$ be a coarse Lipschitz embedding.  As $f$ is a coarse Lipschitz embedding, there exists $C>0$ such that $\rho_f(t)\geq C^{-1}t-C$, and $\omega_{f_2}(t)\leq Ct+C$, for all $t>0$. 

Fix $k\in\N$, and $\delta\in (0,1)$. By  Lemma \ref{lemmasum}, there exists $x\in (\oplus X)_\mathcal{E}$, $\tau>k$, and a compact subset  $K\subset Y_1$ such that, for any  block sequence $(y_n)_n$ in $B_{(\oplus X)_\mathcal{E}}$, there exists $n_0\in\N$, such that 

$$f_1(x+\tau y_n)\subset K+\delta\tau B_{Y_1}, \ \ \text{ for all }\ \ n>n_0.$$\hfill

As $X$ does not contain $\ell_1$, there exists a normalized weakly null sequence $(x_n)_n$ in $X$, with $\inf_{n\neq m}\|x_n-x_m\|>0$.  As $X$ has the  $r$-Banach-Saks property, there exists $c>1$ (independent of $k$), such that, by taking a subsequence if necessary, we have that

$$\|x_{n_1}+\ldots +x_{n_k}\|\leq ck^{1/r},$$\hfill

\noindent for all $n_1<\ldots <n_k\in\N$. For each $j\in\N$, denote by $x_n^j$ the element of $(\oplus X)_\mathcal{E}$  whose $j$-th $\mathcal{E}$-coordinate is $x_n$, and all other $\mathcal{E}$-coordinates are zero. Define $\varphi_{k,\delta}:G_k(\N)\to (\oplus X)_\mathcal{E}$ as 

$$\varphi_{k,\delta}(n_1,\ldots ,n_k)=x+\frac{\tau}{c} k^{-1/r}(x^{n_1}_{n_1}+\ldots +x^{n_1}_{n_k}),$$\hfill

\noindent for all $(n_1,\ldots ,n_k)\in G_k(\N)$. As $\text{Lip}(f_2\circ\varphi_{k,\delta})\leq 2\tau k^{-1/r}Cc^{-1}+C$,  by Theorem \ref{KR4.2}, there exists $\M_{k,\delta}\subset \N$ such that

$$\text{diam}(f_2\circ\varphi_{k,\delta}(G_k(\M_{k,\delta})))\leq 2K\tau k^{1/l-1/r}Cc^{-1}+KCk^{1/l},$$\hfill

\noindent for some $K>0$ independent of $k$ and $\delta$.

Similarly as in the proof of Theorem \ref{jointlowerupper}, by passing to a further subsequence if necessary, we have that, for all $(n_1,\ldots ,n_k)\in G_k(\M_{k,\delta})$, 

$$f_1\circ \varphi_{k,\delta}(n_1,\ldots ,n_k)\subset K+\delta\tau B_{Y_1}.$$\hfill

\noindent Therefore, by taking a further subsequence, we can assume that $\text{diam}(f_1\circ \varphi_{k\delta}(\M_{k,\delta}))\leq 3\delta \tau$ (see Lemma 4.1 of \cite{KR}).

Without loss of generality, we may assume that $\M_{k+1,\delta}\subset \M_{k,\delta}$, for all $k\in\N$. Let $\M_\delta\subset \N$ be a diagonalization of  $(\M_k)_k$. Using the fact that  $X$ has the $r$-co-Banach-Saks property to the weakly null sequence $(x_{n_{2j-1}}-x_{n_{2j}})_j\in {\M_\delta}$, we get that, there exists $d>0$ (independent of $k$) such that, for all $k\in\N$, there exists $ n_1<\ldots <n_{2k}\in \M_{k,\delta}$, such that

$$dk^{1/r}\leq \|\sum_{j=1}^k(x_{n_{2j-1}}-x_{n_{2j}})\|.$$\hfill

\noindent  Therefore,  $\text{diam}(f\circ\varphi_{k,\delta}(G_k(\M_\delta)))\geq \tau dC^{-1}c^{-1}-C$, and we get that

$$\tau dC^{-1}c^{-1}-C\leq 2K\tau k^{1/l-1/r}Cc^{-1}+3Ck^{1/l}+ K\delta\tau,$$\hfill

\noindent for all $k\in\N$, and all $\delta\in (0,1)$. As $\tau>k$, we get that 

 $$dC^{-1}c^{-1}-Ck^{-1}\leq 2K k^{1/l-1/r}Cc^{-1}+KCk^{1/l-1}+ 3\delta,$$
\hfill 
 
 \noindent for all $k\in\N$, and all $\delta\in(0,1)$. As $l>r$, by letting $k\to \infty$ and $\delta\to 0$, we get  a contradiction.
\end{proof}

\begin{thm}\label{oqueestoufazendodavida2}
Let $1\leq q< r$ and $1<p<l$. Assume that

\begin{enumerate}[(i)]
\item  $X$ is an asymptotically $r$-uniformly  Banach space with the $r$-co-Banach-Saks property, and it does not contain $\ell_1$, 
\item $\mathcal{E}$ be a weakly null normalized $1$-unconditional basis with the $p$-co-Banach-Saks property, and the $p$-Banach-Saks property,
\item $Y_1$ be a Banach space with a $1$-unconditional basis satisfying a lower $\ell_q$-estimate with constant $1$, and 
\item  $Y_2$ be a reflexive asymptotically $l$-uniformly smooth Banach space. 
\end{enumerate}

\noindent  Then $(\oplus X)_\mathcal{E}$ does not coarse Lipschitz embed into $Y_1\oplus Y_2$.  
\end{thm}

\begin{proof}
Let $f=(f_1,f_2):(\oplus X)_\mathcal{E}\to Y_1\oplus_1 Y_2$ be a coarse Lipschitz embedding. As $f$ is a coarse Lipschitz embedding, there exists $C>0$ such that $\rho_f(t)\geq C^{-1}t-C$, and $\omega_{f_2}(t)\leq Ct+C$, for all $t>0$. 

Fix $k\in \N$ and $\delta\in (0,1)$. As $X$ does not contain $\ell_1$, there exists a normalized weakly null sequence $(x_n)_n$ in $X$, with $\inf_{n\neq m}\|x_n-x_m\|>0$. For each $j\in\N$, denote by $x_n^j$ the element of $(\oplus X)_\mathcal{E}$  whose $j$-th $\mathcal{E}$-coordinate is $x_n$, and all other $\mathcal{E}$-coordinates are zero.  As $\mathcal{E}$ has both the $p$-co-Banach-Saks property and the $p$-Banach-Saks property, by taking a subsequence of $\mathcal{E}$ if necessary, we can assume that 

\begin{align}\label{cobanb}
dk^{1/p}\leq \|\sum_{j=1}^k(x^j_{n_{2j-1}}-x^j_{n_{2j}})\|\ \ \text{ and }\ \ \|x^1_{n_1}+\ldots +x^k_{n_k}\|\leq c k^{1/p},
\end{align}\hfill

\noindent for all $n_1<\ldots <n_{2k}\in\N$, where $c,d$ are positive constants (independent of $k$).

Let $(\oplus_{i=1}^k X)_\mathcal{E}$ be the subspace of $(\oplus X)_\mathcal{E}$ whose $j$-th coordinates are zero for all $j>k$, and define $\varphi:G_k(\N)\to (\oplus_{i=1}^k X)_\mathcal{E}$ as 

$$\varphi(n_1,\ldots ,n_k)=x+\frac{\tau}{c} k^{-1/p}(x^{1}_{n_1}+\ldots +x^{k}_{n_k}),$$\hfill

\noindent for all $(n_1,\ldots ,n_k)\in G_k(\N)$. 

As $(\oplus_{i=1}^k X)_{\ell_\infty}$ is asymptotically $p$-uniformly smooth, Lemma \ref{lemma} and Remark \ref{equivnorm} gives us that there exists  $x\in (\oplus_{i=1}^k X)_\mathcal{E}$, $\tau>k$, and a compact subset  $K\subset Y_1$ such that, for any weakly null sequence $(y_n)_n$ in $B_{(\oplus_{i=1}^k X)_\mathcal{E}}$, there exists $n_0\in\N$, such that 

$$f_1(x+\tau y_n)\subset K+\delta\tau B_{Y_1}, \ \ \text{ for all }\ \ n>n_0.$$\hfill

\noindent So, similarly as in the proof of Theorem \ref{jointlowerupper}, there exists $\M_0\subset \N$ such that, for all $(n_1,\ldots ,n_k)\in G_k(\M_0)$, we have

$$f_1\circ \varphi(n_1,\ldots ,n_k)\subset K+\delta\tau B_{Y_1}.$$\hfill

\noindent Therefore, by taking a further subsequence $\M_1\subset \M_0$, can assume that $\text{diam}(f_1\circ \varphi(\M_1))\leq 3\delta \tau$.

As $\text{Lip}(f_2\circ\varphi)\leq 2\tau k^{-1/p}Cc^{-1}+C$,  by Theorem \ref{KR4.2}, there exists $\M_2\subset \M_1$ such that

$$\text{diam}(f_2\circ\varphi(G_k(\M_2)))\leq 2K\tau k^{1/l-1/p}Cc^{-1}+KCk^{1/l},$$\hfill

\noindent for some $K>0$ independent of $k$ and $\delta$.

By Equation \ref{cobanb}, we have that   $\text{diam}(f\circ\varphi(G_k(\M_2)))\geq \tau dC^{-1}c^{-1}-C$, and we conclude  that

 $$dC^{-1}c^{-1}-Ck^{-1}\leq 2K k^{1/l-1/p}Cc^{-1}+KCk^{1/l-1}+ 3\delta,$$
\hfill 
 
 \noindent for all $\delta\in(0,1)$, and all $k\in\N$. As $l>p$, this gives us a contradiction.
\end{proof}

\begin{proof}[Proof of Theorem \ref{cortsisum}]
Theorem \ref{oqueestoufazendodavida} gives us the case $p<q$, and Theorem \ref{oqueestoufazendodavida2} the case $q<p$. 
\end{proof}

\end{document}